\documentclass{article}
\usepackage[left=2.0cm, right=2.0cm, top=2.3cm, bottom=2.6cm]{geometry}
\usepackage{epsfig}
\usepackage{amssymb}
\usepackage{amsmath}
\usepackage{graphicx}
\usepackage{relsize}
\usepackage{xcolor}
\usepackage{cases}
\usepackage{bm}
\usepackage{multirow}
\usepackage{caption}
\usepackage{colortbl}
\usepackage{amsthm}
\usepackage{mathrsfs}
\usepackage{algpseudocode,algorithmicx,algorithm}
\def\ep{{\bm \epsilon}}
\def\sig{{\bm \sigma}}

\def\bU{{\boldsymbol U}}
\def\bV{{\boldsymbol V}}
\def\bu{{\boldsymbol u}}
\def\bv{{\boldsymbol v}}
\def\C{\Lambda}

\renewcommand{\div}{\operatorname{div}}

\def\ep{\bm \epsilon}
\def\sig{\bm \sigma}
\def\bU{\boldsymbol U}
\def\bV{\boldsymbol V}
\def\bP{\boldsymbol P}
\def\bff{\boldsymbol f}
\def\bu{\boldsymbol u}
\def\bv{\boldsymbol v}
\def\bp{\boldsymbol p}
\def\bw{\boldsymbol w}
\def\bo{\boldsymbol 0}

\def\bz{\boldsymbol z}
\def\bq{\boldsymbol q}

\def\divv{\text{div\,}}

\def\Divv{\text{Div\,}}

\renewcommand{\div}{\operatorname{div}}

\def\divv{\text{div$\,$}}
\def\ep{\bm \epsilon}
\def\sig{\bm \sigma}
\def\bp{\boldsymbol{p}}
\def\bP{{\boldsymbol P}}
\def\bq{\boldsymbol{q}}

\def\bu{{\boldsymbol u}}
\def\bU{{\boldsymbol U}}
\def\bff{{\boldsymbol f}}
\def\bo{{\boldsymbol 0}}
\def\bv{{\boldsymbol v}}
\def\bV{{\boldsymbol V}}
\def\bz{{\boldsymbol z}}
\def\bw{{\boldsymbol w}}

\def\be{\boldsymbol e}
\definecolor{viol}{rgb}{0.75,0.15,0.95}
\definecolor{orange}{rgb}{0.0,0.0,0.0}
\definecolor{pinkred}{rgb}{0.0,0.0,0.0}
\definecolor{pinkredd}{rgb}{0.0,0.0,0.0}
\newtheorem{theorem}{Theorem}

\newtheorem{lemma}[theorem]{Lemma}

\begin{document}
% Do not eliminate the following size-changing command, please:
\large

%\begin{flushleft}
%
%% Title:
%{\sffamily\bfseries \Large }\\
%\vspace*{3mm}
%% 
%\end{flushleft}

\title{Parameter-robust convergence analysis of fixed-stress split iterative method for multiple-permeability poroelasticity systems}
\author{Qingguo Hong, Johannes Kraus, Maria Lymbery, Mary F. Wheeler}
\maketitle
\begin{abstract}
We consider flux-based multiple-porosity/multiple-permeability poroelasticity systems describing mulitple-network flow and deformation in a 
poro-elastic medium, also referred to as MPET models. The focus of the paper is on the convergence analysis of the fixed-stress split iteration, a commonly used coupling technique for the 
flow and mechanics equations defining poromechanical systems. We formulate the fixed-stress split method in this context and prove its linear convergence. 
The contraction rate of this fixed-point iteration does not depend on any of the physical parameters appearing in the model.  
This is confirmed by numerical results which further demonstrate the advantage of the fixed-stress split scheme over a fully implicit method relying 
on norm-equivalent preconditioning.

\end{abstract}
{\small\textbf{Keywords:} multiple-porosity/multiple-permeability poroelasiticity, MPET system, fixed-stress split iterative coupling, convergence analysis}

\section{Introduction}
Double-porosity poroelasticity models have been used to describe the motion of liquids in fissured rocks as early as in~\cite{Barenblatt1960basic}. 
As a generalization of Biot's theory of consolidation,~\cite{Biot1941general,Biot1955theory}, they have been further extended in the framework of
multiple-network poro\-elastic theory (MPET) where the deformable elastic matrix is permeated by more than two fluid networks 
with differing porosities and permeabilities. The latter find important applications in biophysics and medicine, 
see~\cite{TullyVentikos2011cerebral,Chou2016afully,Guo_etal2018subject-specific,Vardakis2016investigating}.

In a bounded Lipschitz domain $\Omega \subset \mathbb R^d$, $d=2,3$, the mathematical model is described by the MPET system:
\begin{subequations}\label{eq:MPET}
\begin{align}
-\divv \sig + \sum_{i=1}^{n}\alpha_i \nabla p_i  &= \bff~~ \text{in}~~ \Omega\times (0,T),\label{MPET1}\\
\bv_i &= -K_i \nabla p_i\;\; \text{in}~~ \Omega\times (0,T),%~~ i=1,\ldots,n,  
\label{MPET2} \\
-\alpha_i \divv \dot{\bu} - \divv \bv_i - c_{p_{i}} \dot{p}_{i} - \sum_{\substack{j=1\\j\neq i}}^{n}\beta_{ij} (p_i-p_j) &=g_i\;\;\text{in}~~ \Omega\times (0,T),%~~  i=1,\ldots,n,
\label{MPET3}
\end{align}
\end{subequations}
where~\eqref{MPET1} and~\eqref{MPET2} are for $i=1,\ldots,n$. Here
\begin{equation}
	\sig = 2\mu \ep(\bu) + \lambda \text{div}(\bu)\bm I \quad \text{and}\quad \ep(\bu)  = \frac{1}{2}(\nabla \bu + (\nabla \bu)^T), \label{constitutive_compatibility}
\end{equation}
denote the effective stress and the strain tensor respectively and the Lam\'e parameters $\lambda$ and $\mu$ are expressed in terms of the modulus of
elasticity $E$ and the Poisson ratio $\nu\in[0,1/2)$ by
$\lambda:=\frac{\nu E}{(1+\nu)(1-2\nu)}$, $\mu:=\frac{E}{2(1+\nu)}$.
The displacement $\bu$, fluxes $\bv_i$ and corresponding pressures $p_i$, $i=1,\dots,n$, are the unknown physical quantities. 

The constants $\alpha_i$ in~\eqref{MPET1} are known as Biot-Willis parameters while $\bff$ represents the body force density.  
The hydraulic conductivity tensors $K_i$ in~\eqref{MPET2} are defined as the permeability divided by the viscosity of the $i$-the network.
The constants $c_{p_{i}}$ in~\eqref{MPET3} denote the constrained specific storage coefficients, see e.g.~\cite{Showalter2010poroelastic} and the references therein. 
The network transfer coefficients $\beta_{ij}$ couple the network pressures and hence $\beta_{ij}=\beta_{ji}$. 
Fluid extractions or injections enter the system via the source terms $g_i$ in~\eqref{MPET3}.

The system~\eqref{eq:MPET} is well posed under proper boundary and initial conditions. For stability reasons, this system is discretized in time by an implicit method. 
This creates 
a coupled static problem in each time step. The latter can be solved fully implicit, using a loose or explicit coupling, or an iterative coupling. In general, the loosely or 
explicitly coupled approach is less accurate than the fully implicit one which, however, is normally more computationally expensive. Iterative coupling is a commonly used 
alternative to avoid the disadvantages of the aforementioned approaches. The most popular procedures in this category are the undrained split, the 
fixed-stress split, the drained split and the fixed-strain split iterative methods. As shown in~\cite{Kim2011stability}, in contrast to the drained split and the fixed-strain split methods, 
the undrained split and fixed-stress split methods are unconditionally stable. 

Convergence estimates and the rate of convergence for latter methods have been derived 
in~\cite{Mikelic2013convergence} for the quasi-static Biot system. The convergence and error analysis of an iterative coupling scheme for solving a fully discretized Biot system 
based on the fixed-stress split 
has been provided in~\cite{almani2017convergence}. Linear convergence in energy norms of a variant of the fixed-stress split iteration applied to 
heterogenous media has been shown in~\cite{Both2017robust} for linearized Biot's equations. 

Other variants of the fixed-stress split iterative scheme include a two-grid 
algorithm in which the flow subproblem of the Biot system is solved on a fine grid whereas the poromechanics subproblem is solved on a 
coarse grid, see~\cite{dana2018convergence2}, 
or the multi-rate fixed-stress split iterative scheme which exploits different time scales for the mechanics and flow problems by taking several finer time steps for flow within one 
coarse time step for the mechanics of the system, see~\cite{Almani2016convergence}. 

The fixed-stress split scheme has also been successfully applied and proved convergent for space-time finite element approximations of the quasi-static Biot system,
cf.~\cite{Bause2017space}.
In the context of unsaturated materials, it can be used for linearization of non-linear poromechanics problems.
When combined with Anderson acceleration, as shown in~\cite{Both2018Anderson}, this yields a highly efficient method. 
The optimization of the stabilization parameter 
that serves the acceleration of the fixed-stress iterative method is considered for the Biot problem in the two-field formulation in~\cite{Storvik2018on}. 

In this paper we propose a fixed-stress split method for the MPET system. We prove its linear convergence and, furthermore, show 
with a proper choice for the stabilization parameter  
that the rate of convergence is independent of the physical parameters in the model. 
These theoretical findings are also tested computationally. 
The obtained numerical results support the proven convergence rate estimate and demonstrate the precedence of the fixed-stress split iterative method over 
the MinRes algorithm with norm-equivalent preconditioning.

The remainder of the paper is structured as follows. In Section~\ref{properties} we introduce notation and recall some important stability properties of the flux-based MPET system, 
see~\cite{Hong2018conservativeMPET}, also~\cite{HongKraus2017parameter} for the special case of Biot's system, which are to be used later. 
Section~\ref{fixed-stress} contains the main contribution of the paper. There, the fixed-stress algorithm for the MPET system is formulated and a parameter-robust convergence rate estimate 
proven. Numerical tests for the proposed fixed-stress split iterative coupling scheme are presented in Section~\ref{numerics}. Section~\ref{conclusions} gives concluding remarks.
\section{Properties of the flux-based MPET problem}\label{properties}

Firstly, we present the operator form of the MPET equations~\eqref{eq:MPET}.
After imposing boundary and initial conditions to this system
to obtain a well-posed problem, 
we use the backward Euler method for its time discretization. Subsequently, a static problem in each time step has to be solved 
which with rescaling and proper variable substitutions has the form:
\begin{equation}\label{eq:7}
\mathscr{A}%\begin{bmatrix}
\left[
{\boldsymbol u}^T, % &
{\boldsymbol v}_1^T, % &
\dots , % &
{\boldsymbol v}_n^T, % &
p_1, % &
\dots , % &
p_n
\right]^T
%\end{bmatrix}^T
=
%\begin{bmatrix}
\left[
{\boldsymbol f}^T, % &
{\boldsymbol 0}^T, % &
\dots , % &
{\boldsymbol 0}^T, % &
{g}_1, % &
\dots , % &
{g}_n
\right]^T. 
\end{equation}
Here
\begin{equation}\label{operator:A}
\mathscr{A}:=
\begin{bmatrix}
- {\text{div\,}} {\boldmath \epsilon}  - \lambda \nabla {\text{div\,}}     & 0  & \dots     & \dots&0      &   \nabla & \dots & \dots & \nabla    \\
\\
0      & R_1^{-1}I &     0      & \dots  &0         &  \nabla  &  0      & \dots  &0 \\
\vdots &  0           &     \ddots &     & \vdots       &  0           &     \ddots &     & \vdots \\
\vdots &      \vdots  &            &\ddots  & 0            &       \vdots  &            &\ddots  & 0 \\
0      &     0        &  \dots     & 0      &R_n^{-1}I&  0        &  \dots     & 0   & \nabla\\
\\
- {\text{div\,}} &-  {\text{div\,}}&0 &\dots   &  0            &\tilde\alpha_{11}I& \alpha_{12} I & \dots& \alpha_{1n}I\\
\vdots                 &   0       &   \ddots        &                 & \vdots      &  \alpha_{21}I  & \ddots & & \alpha_{2n}I\\
\vdots                 &  \vdots  &              & \ddots              & 0    &  \vdots &  & \ddots & \vdots\\
- {\text{div\,}} &   0    & \dots &   0    &-  {\text{div\,}}  & \alpha_{n1}I  &  \alpha_{n2}I &   \dots   & \tilde\alpha_{nn} I  \\     
\end{bmatrix}
\end{equation}
is the rescaled operator, $\tau$ the time step size and
$$R^{-1}_i := \tau^{-1}K_i^{-1}\alpha_i^2,\quad \alpha_{p_i} := \frac{c_{p_i}}{\alpha_i^2},\quad \beta_{ii} := \sum_{\substack{j=1\\j\neq i}}^{n}\beta_{ij},\quad
{\alpha}_{ij} := \frac{\tau \beta_{ij}}{\alpha_i \alpha_j},\quad \tilde\alpha_{ii}:=- \alpha_{p_i} - \alpha_{ii}$$ for $i,j=1,\cdots,n$.
General and plausible assumptions for the scaled parameters, namely, 
\begin{align}\label{parameter:range}
{\lambda > 0,}\quad  R^{-1}_1 ,\dots, R^{-1}_n > 0 ,\quad \alpha_{p_1},\dots,\alpha_{p_n} \ge 0,\quad \alpha_{ij}\ge 0, ~~~i, j=1,\dots, n
\end{align}
are made.

\subsection{Preliminaries and notation}
Denote ${\boldsymbol v}^T:=({\boldsymbol v}_1^T,\dots,{\boldsymbol v}_n^T)$, ${\boldsymbol z}^T:=({\boldsymbol z}_1^T,\dots,{\boldsymbol z}_n^T)$, 
${\boldsymbol p}^T:=(p_1,\dots,p_n)$, ${\boldsymbol q}^T:=(q_1,\dots,q_n)$ where 
${\boldsymbol v}, {\boldsymbol z} \in {\boldsymbol V}={\boldsymbol V}_1\times \dots\times {\boldsymbol V}_n$,
${\boldsymbol p}, {\boldsymbol q} \in {\boldsymbol P}=P_1\times \dots\times P_n$ and ${\boldsymbol U} \hspace{-0.4ex}=\hspace{-0.4ex} \{ {\boldsymbol u} \in H^1(\Omega)^d: {\boldsymbol u} = {\boldsymbol 0}
\text{ on } \Gamma_{{\boldsymbol u},D} \},
{\boldsymbol V}_i \hspace{-0.4ex}=\hspace{-0.4ex} \{ {\boldsymbol v}_i\in H({\text{div}},\Omega):
{\boldsymbol v}_i \cdot {\boldsymbol n} =0  \text{ on } \Gamma_{p_i,N} \}$, $P_i = L^2(\Omega)$, and $P_i = L^2_0(\Omega)$
if $\Gamma_{{\boldsymbol u},D} = \Gamma=\partial \Omega$.
 
The weak formulation of system~\eqref{eq:7} reads as: 
Find $({\boldsymbol u}; {\boldsymbol v};{\boldsymbol p}) \in {\boldsymbol U}\times {\boldsymbol V}\times {\boldsymbol P}$, 
such that for any
$({\boldsymbol w}; {\boldsymbol z};{\boldsymbol q})  \in {\boldsymbol U}\times {\boldsymbol V}\times {\boldsymbol P}$ there hold
\begin{subequations}\label{MPET_weak}
\begin{eqnarray}
	({\boldmath \epsilon}({\boldsymbol u}), {\boldmath \epsilon} ({\boldsymbol w}))
	+ \lambda  ({\text{div\,}}{\boldsymbol u},{\text{div\,}} {\boldsymbol w}) -\sum_{i=1}^{n} (p_i,{\text{div\,}} {\boldsymbol w})
	&=& ({\boldsymbol f},{\boldsymbol w}),\\
	 (R^{-1}_i{\boldsymbol v}_i,{\boldsymbol z}_i) {-} (p_i,{\text{div\,}} {\boldsymbol z}_i) &=& 0, \quad\qquad\, i=1,\dots,n,\\
	-({\text{div\,}}{\boldsymbol u},q_i) - ({\text{div\,}}{\boldsymbol v}_i,q_i)  
	%-(\alpha_{p_i}+{\alpha}_{ii}) (p_i,q_i)
	+ \tilde\alpha_{ii} (p_i,q_i)
	+ \sum_{\substack{{j=1}\\j\neq i}}^{n}  \alpha_{ij}(p_j,q_i) &=& (g_i,q_i) ,\quad  i=1,\dots,n,
\end{eqnarray}
\end{subequations}
or,  equivalentely,
$\mathcal{A}(({\boldsymbol u};{\boldsymbol v};{\boldsymbol p}),({\boldsymbol w};{\boldsymbol z};{\boldsymbol q}))=
F({\boldsymbol w};{\boldsymbol z};{\boldsymbol q})$ for $({\boldsymbol w};{\boldsymbol z};{\boldsymbol q})
\in {\boldsymbol U}\times {\boldsymbol V}\times {\boldsymbol P}$, 
where $F({\boldsymbol w};{\boldsymbol z};{\boldsymbol q})= ({\boldsymbol f},{\boldsymbol w})+\sum\limits_{i=1}^n(g_i,q_i)$
and
 \begin{align*}
&\mathcal{A}((\bu;\bv; \bm p),(\bw;\bz;\bm q))= (\ep(\bu), \ep(\bw))+ \lambda  (\divv\bu,\divv \bw) -\sum_{i=1}^{n} (p_i,\divv \bw) 
+\sum_{i=1}^{n} (R^{-1}_i\bv_i,\bz_i) \\&- \sum_{i=1}^{n}(p_i,\divv \bz_i)
 -\sum_{i=1}^{n}(\divv\bu,q_i)-\sum_{i=1}^{n}(\divv\bv_i,q_i)  -\sum_{i=1}^{n}(\alpha_{p_i}+\alpha_{ii})(p_i,q_i) 
+ \sum_{i=1}^{n}\sum_{\substack{j=1\\j\neq i}}^{n} \alpha_{ji} (p_j,q_i) \\
& = (\ep(\bu), \ep(\bw))+ \lambda  (\divv\bu,\divv \bw) -\sum_{i=1}^{n} (p_i,\divv \bw) 
+\sum_{i=1}^{n} (R^{-1}_i\bv_i,\bz_i) - \sum_{i=1}^{n}(p_i,\divv \bz_i) \\&
 -\sum_{i=1}^{n}(\divv\bu,q_i) - ({\text{Div\,}} {\boldsymbol v},\boldsymbol q)-((\Lambda_1+\Lambda_2)\boldsymbol p,\boldsymbol q).
 \end{align*}
Here we have denoted $({\text{Div\,}} {\boldsymbol v})^T:= ({\text{div\,}} {\boldsymbol v}_1, \ldots, {\text{div\,}} {\boldsymbol v}_n)$ and
$$
\begin{array}{ll}
\Lambda _{1}:=  
\begin{bmatrix}
\alpha_{11} & -\alpha_{12} & \dots &-\alpha_{1n}  \\
-\alpha_{21} & \alpha_{22} & \dots &-\alpha_{2n}  \\
\vdots & \vdots & \ddots & \vdots  \\
-\alpha_{n1} & -\alpha_{n2} & \dots &\alpha_{nn}  
\end{bmatrix},&
\Lambda _2:=
\begin{bmatrix}
{\alpha_{p_1}} &0&\dots &0\\
0&{\alpha_{p_2}}&\dots &0\\
\vdots&\vdots &\ddots&\vdots\\
0&0&\dots&  {\alpha_{p_n}}
\end{bmatrix}.
\end{array}
$$
Furthermore, define $R^{-1} := \max\{ R_1^{-1} , \dots,R_n^{-1} \}$, $\lambda_0:=\max\{1, \lambda\}$ and the $n\times n$ matrices
$$
\begin{array}{ll}
 \Lambda_{3}:=
\begin{bmatrix}
R &0&\dots &0\\
0&R&\dots &0\\
\vdots& \vdots&\ddots&\vdots\\
0&0&\dots & R
\end{bmatrix},&
 \Lambda_{4}:=
\begin{bmatrix}
\frac{1}{\lambda_0} &\dots& {\dots} & \frac{1}{\lambda_0}\\
\vdots& & &\vdots\\
\vdots& & &\vdots\\
\frac{1}{\lambda_0}&\dots & {\dots}&\frac{1}{\lambda_0}
\end{bmatrix}
%, 
%\qquad \Lambda  := {\sum_{i=1}^{4}\Lambda _i},
\end{array}
$$
that are used later in the convergence analysis of the fixed-stress coupling iteration.
It is easy to show that $\Lambda_i$ are symmetric positive semidefinite (SPSD) for $i=1,2,4$ 
while $\Lambda_3$ is symmetric positive definite (SPD). 

Moreover, we denote $$\Lambda  := {\sum\limits_{i=1}^{4}\Lambda _i}$$ which obviously is an SPD matrix and therefore, can be used to 
define the parameter-matrix-dependent  norms 
$\|\cdot\|_{\boldsymbol U}$, $\|\cdot\|_{\boldsymbol V}$, $\|\cdot\|_{\boldsymbol P}$ induced by the inner products:
\begin{subequations}\label{norms}
\begin{align}
({\boldsymbol u},{\boldsymbol w})_{\boldsymbol U}&= ({\boldmath \epsilon}({\boldsymbol u}),{\boldmath \epsilon}({\boldsymbol w}))
+ \lambda({\text{div\,}} {\boldsymbol u},{\text{div\,}} {\boldsymbol w}),\label{norms-u}\\
({\boldsymbol v}, {\boldsymbol z})_{\boldsymbol V}&= \sum_{i=1}^n(R_i^{-1}{\boldsymbol v}_i,{\boldsymbol z}_i)
+ (\Lambda^{-1}  {\text{Div\,}} {\boldsymbol v} ,{\text{Div\,}} {\boldsymbol z}),\label{norms-v}\\
({\boldsymbol p},{\boldsymbol q})_{\boldsymbol P}&= (\Lambda {\boldsymbol p},{\boldsymbol q})\label{norms-p}.
\end{align}
\end{subequations}
As shown in~\cite{Hong2018conservativeMPET}, these norms are crucial to show the parameter-robust stability of the MPET system. 
\subsection{Stability properties}

The following inf-sup conditions for the spaces ${\boldsymbol U}, {\boldsymbol V}, {\boldsymbol P}$ are assumed to be fulfilled in the analysis presented in this paper:   
\begin{eqnarray}
\inf_{q\in P_i} \sup_{{\boldsymbol v} \in {\boldsymbol V}_i} \frac{({\rm div} {\boldsymbol v}, q)}{\|{\boldsymbol v}\|_{\rm div}\|q\|}
\geq \beta_d, \quad i=1,\dots,n, \label{eq:inf_sup_d} \\
\inf_{(q_1,\cdots,q_n)\in P_1\times\cdots\times P_n}
\sup_{{\boldsymbol u}\in {\boldsymbol U}}
\frac{\left({\rm div} {\boldsymbol u}, \sum\limits_{i=1}^n q_i\right)}{\|{\boldsymbol u}\|_1\left\|\sum\limits_{i=1}^n q_i\right\|} \geq \beta_s  
\label{eq:inf_sup_s}
\end{eqnarray}
for some constants $\beta_d > 0$ and $\beta_s > 0$, see \cite{ Brezzi1974existence,Boffi2013mixed}.  
Then from \cite{Hong2018conservativeMPET}, we know that the MPET problem~\eqref{MPET_weak} is uniformly well-posed, namely 
the three assertions in Theorem~\ref{continu_stability} hold:
\begin{theorem}\label{continu_stability}
~
\begin{itemize}
\item[(i)] 
There exists a positive constant $C_{b}$ independent of  the parameters $\lambda$, $R_i^{-1}$, $\alpha_{p_i}$, ${\alpha}_{ij}$, 
$i,j \in \{1,\dots,n\}$ and the network scale~$n$ such that the inequality 
\begin{equation*}
|\mathcal A((\bu;\bv;\bm p),(\bw;\bz;\bm q))|\le C_b (\|\bm u|_{\bm U}+\|\bm v\|_{\bm V}
+\|\bm p\|_{\bm P})  (\|\bm w\|_{\bm U}+\|\bm z\|_{\bm V}+\|\bm q\|_{\bm P})
\end{equation*}
holds true for any $(\bm u; \bm v;\bm p)\in \boldsymbol U\times \boldsymbol V\times \boldsymbol P , (\bm w; \bm z; \bm q)\in \boldsymbol U\times \boldsymbol V\times \boldsymbol P$.
\item[(ii)]  There is a constant $\omega> 0$ independent of the 
parameters $\lambda,R_i^{-1}, \alpha_{p_i}, {\alpha}_{ij}$, $i,j \in \{1,\dots,n\}$ and the number of networks $n$ such that 
\begin{align}\label{stability} 
\inf_{({\boldsymbol u};{\boldsymbol v};{\boldsymbol p})\in {\boldsymbol X}}
\sup_{({\boldsymbol w};{\boldsymbol z};{\boldsymbol q})\in {\boldsymbol X}}
\frac{\mathcal{A}(({\boldsymbol u};{\boldsymbol v};{\boldsymbol p}),({\boldsymbol w};{\boldsymbol z};{\boldsymbol q}))}
{( \|{\boldsymbol u}\|_{{\boldsymbol U}}
+\|{\boldsymbol v}\|_{{\boldsymbol V}}+ \|{\boldsymbol p}\|_{{\boldsymbol P}})( \|{\boldsymbol w}\|_{{\boldsymbol U}}
+ \|{\boldsymbol z}\|_{{\boldsymbol V}}+ \|{\boldsymbol q}\|_{{\boldsymbol P}})} \geq \omega,
\end{align}
where ${\boldsymbol X}:={\boldsymbol U} \times {\boldsymbol V}\times {\boldsymbol P}$.
\item[(iii)]  The MPET system~\eqref{MPET_weak} has a unique solution $(\boldsymbol u; \boldsymbol v;\bp)\in \boldsymbol U\times \boldsymbol V\times \boldsymbol P$ and the following stability estimate holds:
%Let $(\boldsymbol u; \boldsymbol v;\bp)\in \boldsymbol U\times \boldsymbol V\times P$ be the solution
%of~\eqref{MPET_weak}. Then there holds the estimate
\begin{equation}\label{MPET_stability_estimate}
\|\boldsymbol u\|_{\boldsymbol U}+\|\boldsymbol  v\|_{\boldsymbol V}+\|\boldsymbol p\|_{\boldsymbol P}\leq C_1 (\|\boldsymbol f\|_{\boldsymbol U^*}
+\|\bm g\|_{\bm P^*}),
\end{equation} 
where $C_1$ is a positive constant independent of the parameters $\lambda,R_i^{-1}, \alpha_{p_i}, {\alpha}_{ij}, i,j \in \{1,\dots,n\}$ and the network scale $n$, 
and
$\|\boldsymbol f\|_{\boldsymbol U^*}=
\sup\limits_{\boldsymbol w\in \boldsymbol U}\frac{(\boldsymbol f, \boldsymbol w)}{\|\boldsymbol w\|_{\boldsymbol U}},~~ \|\boldsymbol g\|_{\boldsymbol P^*}=
\sup\limits_{\boldsymbol q\in \boldsymbol P}\frac{(\boldsymbol g,\boldsymbol q)}
{\|\boldsymbol q\|_{\boldsymbol P}}=\|\Lambda^{-\frac{1}{2}} \boldsymbol g\|.
%%\quad \boldsymbol g^T=(g_1,\cdots,g_n).
$
\end{itemize} 
\end{theorem}

\subsection{A norm equivalent preconditioner}

Consider the block-diagonal operator
\begin{equation}\label{Preconditioner:B}\nonumber
\mathscr{B}:=\left[\begin{array}{ccc}
\mathscr{B}_{\bm u}^{-1} & \bm 0 & \bm 0 \\
\bm 0 & \mathscr{B}_{\bm v}^{-1}& \bm 0 \\
\bm 0& \bm 0& \mathscr{B}_{\bm p}^{-1}  
\end{array}
\right], \quad \text{where} \quad \mathscr{B}_{\bm u}
=-\divv \ep-\lambda \nabla \divv,
\end{equation}
%where 
%$$
%{\color{viol}\mathscr{B}}_{\bm u}
%=-\divv \ep-\lambda \nabla \divv,
%$$
\begin{align*}
\mathscr{B}_{\bm v}&=  
\begin{bmatrix}
R_1^{-1}I&0  & \dots &0  \\
0 &R_2^{-1} I& \dots &0 \\
\vdots &\vdots  & \ddots & \vdots  \\
0& 0 & \dots &R_n^{-1}I
\end{bmatrix}
-
\begin{bmatrix}
\tilde{\gamma}_{11}\nabla{\rm div} &\tilde{\gamma}_{12}\nabla{\rm div} & \dots &\tilde{\gamma}_{1n}\nabla{\rm div}\\
\tilde{\gamma}_{21}\nabla{\rm div} &\tilde{\gamma}_{22}\nabla{\rm div} & \dots &\tilde{\gamma}_{2n}\nabla{\rm div}\\
\vdots &\vdots  & \ddots & \vdots  \\
\tilde{\gamma}_{n1}\nabla{\rm div} &\tilde{\gamma}_{n2}\nabla{\rm div} & \dots &\tilde{\gamma}_{nn}\nabla{\rm div}
\end{bmatrix},
\end{align*}
and
\begin{align*}
\mathscr{B}_{\bm p}&=
\begin{bmatrix}
\gamma_{11}I&\gamma_{12}I& \dots &\gamma_{1n}I\\
\gamma_{21}I&\gamma_{22}I& \dots &\gamma_{2n}I\\
\vdots&\vdots  &\ddots&\vdots\\
\gamma_{n1}I&\gamma_{n2}I& \dots &\gamma_{nn}I\\
\end{bmatrix}.
\end{align*}
Here, $\gamma_{ij}$, $\tilde{\gamma}_{ij}$, $i,j=1,\ldots,n$ are the entries of $\Lambda$ and $\Lambda^{-1}$, respectively.

As substantiated in~\cite{Hong2018conservativeMPET}, the stability results for the operator $\mathscr{A}$ imply that 
the operator~$\mathscr{B}$ %defined in~\eqref{Preconditioner:B}
is a uniform norm-equivalent (canonical) block-diagonal preconditioner % in~\eqref{operator:A},
that is robust with respect to all model and discretization parameters. 

Note that the existence of this canonical uniform block-diagonal preconditioner can be transferred to the discrete level  as long as 
discrete inf-sup conditions analogous to~\eqref{eq:inf_sup_d} and~\eqref{eq:inf_sup_s} are satisfied, cf.~\cite{Hong2018conservativeMPET}.

\section{Fixed-stress method for MPET model}\label{fixed-stress}

In the proposed fixed-stress split iterative coupling scheme for the MPET system, and as for Biot's equations,  we first solve
the flow and then the mechanics problem where, in order to avoid instabilities, a stabilization term is added to the flow 
equation. Note that generalizing the fixed-stress iteration from the Biot to the (flux-based) MPET model is not straightforward due 
to the involvement of $n$ pressures $p_i$ and $n$ fluxes $\bv_i$. 
Our formulation suggests a stabilization that employs the sum of the pressures
%, a quantity related to the total pressure, 
%which later shows itself to be vital for the convergence properties of the scheme. 
which later shows itself to be vital in the convergence analysis 
of the scheme. 

In order to elucidate our approach, we present the fixed-stress splitting scheme for the continuous problem first. Let $\bu^{k}$, $\bv_i^{k}$ and $p_i^{k}$ 
denote the $k$-th fixed-stress iterates for $\bu$, $\bv_i$ and $p_i$ respectively,  $i=1,\ldots,n$. The single rate fixed-stress coupling iteration is 
given by the following algorithm:
%.

\vspace{1ex}
\begin{algorithm}[H]
\textbf{Step a}: Given $\bu^{m}$, we solve for $\bv_i^{m+1}$ and $p_i^{m+1}$ %$1 \le i \le n$,
\begin{equation}\label{fixed_stress_eq1}
\begin{array}{l}
( - \divv\bv_i^{m+1},q_i) -((\alpha_{p_i} + {\alpha}_{ii}) p_i^{m+1},q_i)+\left(\displaystyle\sum_{\substack{{j=1}\\j\neq i}}^{n}{\alpha}_{ij} p_j^{m+1},q_i\right)
 -L\left( \displaystyle\sum_{j=1}^{n} p_j^{m+1},q_i\right)
\\[1ex] \hspace{2.8cm}= (g_i,q_i)-L \left(\displaystyle\sum_{j=1}^{n}p_j^{m},q_i\right)+(\divv \bu^{m},q_i), \quad 1 \le i \le n,
\end{array}
\end{equation}
and 
\begin{equation}\label{fixed_stress_eq2}
(R_i^{-1} \bv_i^{m+1},\bz_i)- (p_i^{m+1},\divv \bz_i) =\bo, \quad 1 \le i \le n.
\end{equation}
\\ \\
\noindent
\textbf{Step b}: Given $\bv_i^{m+1}$ and $p_i^{m+1}$, we solve for $\bu^{m+1}$ 
\begin{equation}\label{fixed_stress_eq3}
( \ep(\bu^{m+1}),\ep(\bw)) + \lambda (\divv \bu^{m+1},\divv \bw)  = (\bff,\bw) +\sum_{i=1}^{n} ( p_i^{m+1},\divv \bw).
\end{equation}
\caption{\textbf{: Fixed-stress coupling iteration for the MPET system}}\label{alg1}
\end{algorithm}

Our main result is formulated in terms of the following quantities:
\begin{subequations}\label{errors_cont_prob}
\begin{eqnarray}
\be_u^k & = & \bm u^k-\bm u\in \bU, \\
\be_{v_i}^k&=&\bm v_i^k-\bm v_i \in \bV_i, \quad i=1,\ldots,n, \\
e_{p_i}^k&=&p_i^k-p_i \in P_i, \quad i=1,\ldots,n, 
\end{eqnarray}
\end{subequations}
denoting the errors of the $k$-th iterates $\bm u^k$, $\bm v^k_i$, $p^k_i$, $i=1,\ldots,n$ generated by Algrorithm~\ref{alg1}. The error block-vectors $\be_v^k$ 
and $\be_p^k$ are defined by 
$(\be_{v}^k)^T = ((\be_{v_1}^k)^T,\ldots,(\be_{v_n}^k))^T$, $(\be_p^k)^T = (e_{p_1}^k,\ldots,e_{p_n}^k)$. 
Since $\bu$, $\bv_i$, $p_i$, $i=1,\ldots,n$ are the exact solutions of~\eqref{MPET_weak}, the error equations
\begin{subequations}
\begin{eqnarray}
\nonumber ( - {\rm Div} \bm e_v^{m+1},\boldsymbol q) -((\Lambda_1+\Lambda_2) \boldsymbol e_p^{m+1},\boldsymbol q)-L\left(\displaystyle\sum_{i=1}^{n}e_{p_i}^{m+1},  \displaystyle\sum_{i=1}^{n} q_i\right)
%\\\\
&\hspace{-1ex} = \hspace{-1ex}& -L \left(\sum_{i=1}^{n} e_{p_i}^{m}, \displaystyle\sum_{i=1}^{n} q_i\right)
\\ \label{error:1} &\hspace{-1ex} \hspace{-1ex} &+\left(\divv \bm e_u^{m},\sum_{i=1}^{n} q_i\right),
\\ \label{error:2}
(R^{-1} \bm e_v^{m+1},\bz)- (\boldsymbol e_{p}^{m+1},{\rm Div} \bz)  &\hspace{-1ex}=\hspace{-1ex}&0,
\\
\label{error:3}
( \ep(\bm e_u^{m+1}),\ep(\bw)) + \lambda (\divv \bm e_u^{m+1},\divv \bw)   &\hspace{-1ex}=\hspace{-1ex}&\left(\displaystyle\sum_{i=1}^{n} e_{p_i}^{m+1},\divv \bw\right),
\end{eqnarray}
\end{subequations}
hold, the latter of which playing a key role in the presented convergence analysis.

Note that in the following
we do not make any further restrictive assumptions on the parameters in~\eqref{MPET_weak} but consider the general situation in which 
only~\eqref{parameter:range} needs to be satisfied. Useful for deriving and defining the tuning parameter $L$ is the constant $c_K$ in the estimate
\begin{equation}\label{ck}
\|\ep(\bw)\|\ge c_{K}\|\divv(\bw)\| \quad \text{for all }\bw \in \bU
\end{equation}
which is used for $\bw =\bm e_u^{m+1}-\bm e_u^{m}$ in the proof of the next Lemma.\footnote{The constant $c_K$ is related to Korn's inequality and, while in 
general not easy to bound tightly from below, can be estimated sufficiently in the discrete setting.}

We perform the convergence analysis in two steps. The first one is the proof of the following lemma.

\begin{lemma}\label{lemma1}
The errors $\be_u^{m+1}$, $\be_v^{m+1}$ and $\be_p^{m+1}$ of the $(m+1)$-st fixed-stress iterate generated by Algorithm~\ref{alg1} for
$L \ge \displaystyle\frac{1}{\lambda+c_K^2}$ 
satisfy the estimate
\begin{equation}\label{eq:lemma1}
\begin{array}{l}
\displaystyle \frac{1}{2}\big(\|\ep(\bm e_u^{m+1})\|^2+ \lambda \| \normalfont{\divv} \bm e_u^{m+1}\|^2\Big)+\|R^{-1/2}\bm e_v^{m+1}\|^2+\|(\Lambda_1+\Lambda_2)^{1/2}\boldsymbol e_{p}^{m+1}\|^2
 \\
%\hspace{0.8cm}
\qquad +\displaystyle \frac{L}{2} \left\| \displaystyle\sum_{i=1}^{n} e_{p_i}^{m+1}\right\|^2
\leq  \frac{L}{2}\left\|\sum_{i=1}^{n} e_{p_i}^{m}\right\|^2, \quad m=0,1,2,\ldots.
\end{array}
\end{equation}
%for $m=0,1,2,\ldots.$
\end{lemma}

\begin{proof}
Setting $\bm z=\bm e_v^{m+1}, \boldsymbol q=-\boldsymbol e_{p}^{m+1}, \bm w=\bm e_{u}^{m+1}$ in~\eqref{error:1}--\eqref{error:3}, it follows that 
\begin{equation}\label{eq_lemma}
\begin{array}{l}
\|\ep(\bm e_u^{m+1})\|^{2}+ \lambda \|\divv \bm e_u^{m+1}\|^{2}+\|R^{-1/2}\bm e_v^{m+1}\|^{2}+\|(\Lambda_1+\Lambda_2)^{1/2}\boldsymbol e_{p}^{m+1}\|^{2}
\\\\ \qquad+L\left(\displaystyle\sum_{i=1}^{n} (e_{p_i}^{m+1}-e_{p_i}^{m}), \sum_{i=1}^{n} e_{p_i}^{m+1}\right)=\left(\divv(\bm e_u^{m+1}-\bm e_u^{m}),\displaystyle \sum_{i=1}^{n} e_{p_i}^{m+1}\right).
\end{array}
\end{equation}
Using the identity 
$$
\left(\sum_{i=1}^{n} (e_{p_i}^{m+1}-e_{p_i}^{m}), \sum_{i=1}^{n} e_{p_i}^{m+1}\right)=\frac{1}{2}\left(\left\|\sum_{i=1}^{n} e_{p_i}^{m+1}-\sum_{i=1}^{n} e_{p_i}^{m}\right\|^{2}+
\left\|\sum_{i=1}^{n} e_{p_i}^{m+1}\right\|^{2}-\left\|\sum_{i=1}^{n} e_{p_i}^{m}\right\|^{2}\right),
$$
equation~\eqref{eq_lemma} can be rewritten as 
\begin{equation}\label{star}
\begin{array}{l}
\|\ep(\bm e_u^{m+1})\|^{2}+ \lambda \|\divv \bm e_u^{m+1}\|^{2}+\|R^{-1/2}\bm e_v^{m+1}\|^{2}+\|(\Lambda_1+\Lambda_2)^{1/2}\boldsymbol e_{p}^{m+1}\|^{2}
\\
\qquad +\displaystyle\frac{L}{2} \left\|\displaystyle\sum_{i=1}^{n} e_{p_i}^{m+1}\right\|^{2} +\displaystyle\frac{L}{2}\left\|\displaystyle\sum_{i=1}^{n} e_{p_i}^{m+1}-\displaystyle\sum_{i=1}^{n} e_{p_i}^{m}\right\|^{2}
\\
=\displaystyle\frac{L}{2}\left \|\displaystyle\sum_{i=1}^{n} e_{p_i}^{m}\right \|^{2}
+\left(\divv(\bm e_u^{m+1}-\bm e_u^{m}), \displaystyle\sum_{i=1}^{n} e_{p_i}^{m+1}\right).
\end{array}
\end{equation}
Now, taking $\bm w=\bm e_u^{m+1}-\bm e_u^{m}$ in~\eqref{error:3} we obtain 
\begin{equation}\label{starstar}
\left(\divv(\bm e_u^{m+1}-\bm e_u^{m}), \displaystyle\sum_{i=1}^{n} e_{p_i}^{m+1}\right)=( \ep(\bm e_u^{m+1}),\ep(\bm e_u^{m+1}-\bm e_u^{m})) + \lambda (\divv \bm e_u^{m+1},\divv (\bm e_u^{m+1}-\bm e_u^{m})), 
\end{equation}
and, substituting~\eqref{starstar} in~\eqref{star}, conclude that 
$$
\begin{array}{l}
\|\ep(\bm e_u^{m+1})\|^{2}+ \lambda \|\divv \bm e_u^{m+1}\|^{2}+\|R^{-1/2}\bm e_v^{m+1}\|^{2}+\|(\Lambda_1+\Lambda_2)^{1/2}\boldsymbol e_{p}^{m+1}\|^{2}
\\
\qquad +\displaystyle\frac{L}{2} \left\|\displaystyle\sum_{i=1}^{n} e_{p_i}^{m+1}\right\|^{2}+\displaystyle\frac{L}{2}\left\|\displaystyle\sum_{i=1}^{n} e_{p_i}^{m+1}-\displaystyle\sum_{i=1}^{n} e_{p_i}^{m}\right\|^{2}
\\
=\displaystyle\frac{L}{2}\left \|\displaystyle\sum_{i=1}^{n} e_{p_i}^{m}\right\|^{2}+( \ep(\bm e_u^{m+1}),\ep(\bm e_u^{m+1}-\bm e_u^{m})) 
+ \lambda (\divv \bm e_u^{m+1},\divv (\bm e_u^{m+1}-\bm e_u^{m}))
\\
 \leq \displaystyle \frac{L}{2}\left\|\displaystyle\sum_{i=1}^{n} e_{p_i}^{m}\right\|^{2}+\frac{1}{2}\Big(\|\ep(\bm e_u^{m+1})\|^{2}
 +\lambda\|\divv \bm e_u^{m+1}\|^{2}\Big)
 \\
 \qquad +\displaystyle\frac{1}{2}\Big(\|\ep(\bm e_u^{m+1}-\bm e_u^{m})\|^{2} + \lambda\|\divv (\bm e_u^{m+1}-\bm e_u^{m})\|^{2}\Big) .
\end{array}
$$
The latter inequality can be expressed equivalently in the form 
\begin{equation}\label{starstarstar}
\begin{array}{l}
\displaystyle \frac{1}{2}\big(\|\ep(\bm e_u^{m+1})\|^{2}+ \lambda \|\divv \bm e_u^{m+1}\|^{2}\Big)+\|R^{-1/2}\bm e_v^{m+1}\|^{2}+\|(\Lambda_1+\Lambda_2)^{1/2}\boldsymbol e_{p}^{m+1}\|^{2}
\\ 
\qquad+\displaystyle \frac{L}{2} \left \|\displaystyle\sum_{i=1}^{n} e_{p_i}^{m+1}\right \|^{2}+\displaystyle\frac{L}{2}\left \|\displaystyle\sum_{i=1}^{n} e_{p_i}^{m+1}-\displaystyle\sum_{i=1}^{n} e_{p_i}^{m}\right \|^{2}
\\
\leq \displaystyle \frac{L}{2}\left \|\displaystyle\sum_{i=1}^{n} e_{p_i}^{m}\right \|^{2}+\frac{1}{2}\Big(\|\ep(\bm e_u^{m+1}-\bm e_u^{m})\|^{2} + \lambda\|\divv (\bm e_u^{m+1}-\bm e_u^{m})\|^{2}\Big) .
\end{array}
\end{equation}
To estimate the last term in~\eqref{starstarstar} consider~\eqref{error:3} again. Subtracting the $m$-th error from the $(m+1)$-st, choosing $\bm w=\bm e_u^{m+1}-\bm e_u^{m}$ and 
applying Cauchy's inequality yields
\begin{equation}\label{square}
\|\ep(\bm e_u^{m+1}-\bm e_u^{m})\|^{2} + \lambda\|\divv (\bm e_u^{m+1}-\bm e_u^{m})\|^{2}
\leq \|\divv(\bm e_u^{m+1}-\bm e_u^{m})\| \left \|\displaystyle\sum_{i=1}^{n} (e_{p_i}^{m+1}-e_{p_i}^{m})\right\|.
\end{equation}
Next, from~\eqref{ck} we have that $\|\ep(\bm e_u^{m+1}-\bm e_u^{m})\|\ge c_{K}\|\divv(\bm e_u^{m+1}-\bm e_u^{m})\|$,
which implies
\begin{equation*}
 (c_K^{2}+\lambda)\|\divv (\bm e_u^{m+1}-\bm e_u^{m})\|
\leq \left \|\displaystyle\sum_{i=1}^{n} e_{p_i}^{m+1}-\displaystyle\sum_{i=1}^{n} e_{p_i}^{m}\right\|,
\end{equation*}
that is,
\begin{equation}\label{estimate:div}
\|\divv (\bm e_u^{m+1}-\bm e_u^{m})\|
\leq  \frac{1}{ \lambda+c_K^2} \left\|\displaystyle\sum_{i=1}^{n} e_{p_i}^{m+1}-\displaystyle\sum_{i=1}^{n} e_{p_i}^{m}\right\|.
\end{equation}
Hence 
\begin{equation}\label{eq:18a}
\begin{array}{rcl}
\|\ep(\bm e_u^{m+1}-\bm e_u^{m})\|^2 + \lambda\|\divv (\bm e_u^{m+1}-\bm e_u^{m})\|^2
&\leq&  \displaystyle\frac{1}{ \lambda+c_K^2} \left\|\displaystyle\sum_{i=1}^{n} (e_{p_i}^{m+1}-e_{p_i}^{m})\right\|^2 \\
&\leq&  L  \left\|\displaystyle\sum_{i=1}^{n} (e_{p_i}^{m+1}-e_{p_i}^{m})\right\|^2.
\end{array}
\end{equation}
Therefore, using~\eqref{eq:18a} in~\eqref{starstarstar}, we obtain
$$
\begin{array}{l}
\displaystyle\frac{1}{2}\big(\|\ep(\bm e_u^{m+1})\|^2+ \lambda \|\divv \bm e_u^{m+1}\|^2\Big)+\|R^{-1/2}\bm e_v^{m+1}\|^2+\|(\Lambda_1+\Lambda_2)^{1/2}\boldsymbol e_{p}^{m+1}\|^2
\\ 
\qquad +\displaystyle\frac{L}{2} \left \|\displaystyle\sum_{i=1}^{n} e_{p_i}^{m+1}\right\|^2+\frac{L}{2}\left\|\displaystyle\sum_{i=1}^{n} e_{p_i}^{m+1}-\displaystyle\sum_{i=1}^{n} e_{p_i}^{m}\right\|^2
%\\
\leq \displaystyle \frac{L}{2}\left \|\displaystyle\sum_{i=1}^{n} e_{p_i}^{m}\right\|^2+\frac{L}{2} \left\|\displaystyle\sum_{i=1}^{n} (e_{p_i}^{m+1}-e_{p_i}^{m})\right\|^2 ,
\end{array}
$$
which completes the proof.
\end{proof}
%\begin{remark}
%Note that for $L=\displaystyle\frac{1}{\lambda+c_K^2}$, the estimate of Lemma~\ref{lemma1} reads
%$$
%\begin{array}{l}
%\displaystyle \frac{1}{2}\big(\|\ep(\bm e_u^{m+1})\|^2+ \lambda \| \normalfont{\divv} \bm e_u^{m+1}\|^2\Big)+\|R^{-1/2}\bm e_v^{m+1}\|^2+\|(\Lambda_1+\Lambda_2)^{1/2}\boldsymbol e_{p}^{m+1}\|^2
%\\
%\qquad +  \displaystyle\frac{1}{2(\lambda+c_K^2)} \left\|\displaystyle\sum_{i=1}^{n} e_{p_i}^{m+1}\right\|^2
%\leq \displaystyle \frac{1}{2(\lambda+c_K^2)}\left\|\displaystyle\sum_{i=1}^{n} e_{p_i}^{m}\right\|^2.
%\end{array}
%$$
%\end{remark}

Using~\eqref{eq:lemma1}, we can prove that $\sum_{i=1}^{n} e_{p_i}^{m}\overset{m \rightarrow \infty}{\longrightarrow} 0$,
which is stated in the following theorem. %Theorem~\ref{theorem1}.

\begin{theorem}\label{theorem1}
Let $c_K$ and $\beta_s$ denote the constants in~\eqref{ck} and~\eqref{eq:inf_sup_s}, respectively.
The single rate fixed-stress iterative method for the static MPET problem~\eqref{MPET_weak} defined in Algorithm~\ref{alg1} is a
contraction that converges linearly for any $L\ge 1/(\lambda+c_K^2)$ independent of the model parameters and the time step
size $\tau$. 
The errors $\be_p^m$ in this case satisfy the inequality
\begin{equation}\label{error_sum_p}
\left\| \sum_{i=1}^n e_{p_i}^{m+1} \right\|^2 \le \normalfont{\text{rate}}^2(\lambda) \left\| \sum_{i=1}^n e_{p_i}^m \right\|^2
\end{equation}
with
\begin{equation}\label{rate_gen}
\normalfont{\text{rate}}^2(\lambda) \le \displaystyle\frac{1}{\frac{L^{-1}}{\beta_s^{-2}+\lambda}+1}.
\end{equation}
For $L =\displaystyle \frac{1}{\lambda+c_{K}^2}$, the convergence factor in~\eqref{error_sum_p}
can be estimated by
\begin{equation}\label{rate_spec}
\normalfont{\text{rate}}^2(\lambda)\le \displaystyle \frac{1}{\frac{\lambda+c_K^2}{\beta_s^{-2}+\lambda}+1}
\le \max\left\{ \displaystyle\frac{\beta_s^{-2}}{c_K^2+\beta_s^{-2}},\displaystyle\frac{1}{2}\right\}.
\end{equation}
\end{theorem}

\begin{proof}
By the Stokes inf-sup condition, we have that for any $\displaystyle\sum_{i=1}^{n} e_{p_i}^{m+1}$ there exists $\bm w_p\in \bm U$ such that 
\begin{equation}\label{4stars}
\divv \bm w_p=\displaystyle\sum_{i=1}^{n} e_{p_i}^{m+1}~~~~\hbox{and}~~~~ \|\ep(\bm w_p)\|\leq \beta_s^{-1} \left\|\displaystyle\sum_{i=1}^{n} e_{p_i}^{m+1}\right\|,
\end{equation}
where $\beta_s$ is the Stokes inf-sup constant in \eqref{eq:inf_sup_s}. Hence,  
$$
\|\ep(\bm w_p)\|^2+\lambda\|\divv \bm w_p\|^2\leq (\beta_s^{-2}+\lambda)\left\|\displaystyle\sum_{i=1}^{n} e_{p_i}^{m+1}\right\|^2.
$$
Taking $\bm w=\bm w_p$ in \eqref{error:3} and using~\eqref{4stars} yields
\begin{equation}
\left\|\displaystyle\sum_{i=1}^{n} e_{p_i}^{m+1}\right \|^2=(\ep(\bm e_u^{m+1}),\ep(\bm w_p)) + \lambda (\divv \bm e_u^{m+1},\divv \bm w_p).
\end{equation}
Now, applying Cauchy's inequality, we obtain
\begin{equation}
\begin{array}{l}
\left\|\displaystyle\sum_{i=1}^{n} e_{p_i}^{m+1}\right\|^2\leq (\|\ep(\bm e_u^{m+1})\|^2+\lambda \|\divv \bm e_u^{m+1}\|^2)^{\frac{1}{2}}(\|\ep(\bm w_p)\|^2+ \lambda \|\divv \bm w_p\|^2)^{\frac{1}{2}}
\\\\
\qquad \leq  (\|\ep(\bm e_u^{m+1})\|^2+\lambda \|\divv \bm e_u^{m+1}\|^2)^{\frac{1}{2}}(\beta_s^{-2}+\lambda)^{\frac{1}{2}}\left\|\displaystyle\sum_{i=1}^{n} e_{p_i}^{m+1}\right \|
\end{array}
\end{equation}
which implies
\begin{equation}\label{eq:ineq1}
(\beta_s^{-2}+\lambda)^{-1}\left\|\displaystyle\sum_{i=1}^{n} e_{p_i}^{m+1}\right\|^2\leq\|\ep(\bm e_u^{m+1})\|^2
+ \lambda \|\divv \bm e_u^{m+1}\|^2.
\end{equation}
Given Lemma~\ref{lemma1} and~\eqref{eq:ineq1}, we therefore obtain 
$$
\begin{array}{l}
\displaystyle\frac{1}{2}(\beta_s^{-2}+\lambda)^{-1}
\left\|\displaystyle\sum_{i=1}^{n} e_{p_i}^{m+1}\right\|^2+\|R^{-1/2}\bm e_v^{m+1}\|^2+\|(\Lambda_1+\Lambda_2)^{1/2}\boldsymbol e_{p}^{m+1}\|^2
\\
\qquad +\displaystyle \frac{L}{2} \left\|\displaystyle\sum_{i=1}^{n} e_{p_i}^{m+1}\right\|^2
\leq \displaystyle \frac{L}{2} \left\|\displaystyle\sum_{i=1}^{n} e_{p_i}^{m}\right\|^2
\end{array}
$$
and hence
$$
\begin{array}{l}
\left(\displaystyle\frac{1}{2\beta_s^{-2}+2\lambda}+\displaystyle\frac{L}{2}\right)
\left\|\displaystyle\sum_{i=1}^{n} e_{p_i}^{m+1}\right\|^2  \leq \displaystyle\frac{L}{2}
\left\|\displaystyle\sum_{i=1}^{n} e_{p_i}^{m}\right\|^2,
\end{array}
$$
or, equivalently,
$$
\begin{array}{l}
\left(\displaystyle\frac{L^{-1}}{\beta_s^{-2}+\lambda}+1\right)\left\|\displaystyle\sum_{i=1}^{n} e_{p_i}^{m+1}\right\|^2  
\leq \left\|\displaystyle\sum_{i=1}^{n} e_{p_i}^{m}\right\|^2
\end{array}
$$
which proves~\eqref{error_sum_p}--\eqref{rate_gen}. Finally, \eqref{rate_spec} follows from~\eqref{rate_gen} by choosing
$L=\displaystyle\frac{1}{\lambda+c_K^2}$ and noting that $\displaystyle \frac{1}{\frac{\lambda+c_K^2}{\beta_s^{-2}+\lambda}+1}$
is a monotone function for $\lambda > 0$.
\end{proof}

Note that $\left\|\sum\limits_{i=1}^{n} e_{p_i}^{m}\right\|$ only defines a seminorm of ${\boldsymbol  e}^m_{p}$ and
Theorem~\ref{theorem1} indicates the convergence rate of ${\boldsymbol  e}_{p}$ in this seminorm.
%$\left\|\sum\limits_{i=1}^{n} e_{p_i}^{m}\right\|$.
It still remains at this point unclear whether  
$\left\|\sum\limits_{i=1}^{n} e_{p_i}^{m}\right\|\rightarrow 0$ guarantees that ${\boldsymbol  e}^m_{p}$ converges to $\boldsymbol 0$.
 
Theorem~\ref{theorem2}, as stated later, clarifies this and demonstrates the uniform convergence of
${\boldsymbol  e}^m_{u}, {\boldsymbol  e}^m_{v}$ and ${\boldsymbol  e}^m_{p}$ for the fixed-stress iterative method 
utilizing the uniform stability results from~\cite{Hong2018conservativeMPET}.
Before we present Theorem~\ref{theorem2}, we introduce the matrices:
\begin{equation*}
 \Lambda_{L}:=
\begin{bmatrix}
L &\dots& {\dots} & L\\
\vdots& & &\vdots\\
\vdots& & &\vdots\\
L&\dots & {\dots}&L
\end{bmatrix}
\end{equation*}
and
$$ 
\C_e := \C+\C_L. 
$$
%Analogously 

Analogous to the assertion of Lemma 1 in \cite{Hong2018conservativeMPET}, the properties of $\C_e$ are as follows in Lemma \ref{etauu}:
%.
\begin{lemma} \label{etauu}
Let $\tilde{\C}= \C_3 + \C_4+\C_L, \tilde{\C}^{-1}=(\tilde b_{ij})_{n\times n}$, then $\tilde{\C}$ is SPD and for any n-dimensional vector $\bm x$, we have
\begin{align}
&(\C_e {\boldsymbol x},{\boldsymbol x}) \geq (\tilde{\C }{\boldsymbol x},{\boldsymbol x})\geq (\C_3 {\boldsymbol x},{\boldsymbol x}),\label{etauu:1}
\\&(\C_e^{-1} {\boldsymbol x},{\boldsymbol x}) \leq (\tilde{\C}^{-1}{\boldsymbol x},{\boldsymbol x})\leq (\C_3^{-1}{\boldsymbol x},{\boldsymbol x})=R^{-1}(\bm x, \bm x).\label{etauu:2}
\end{align}
Also,
{
\begin{align}
&0<\sum_{\substack{i=1}}^n\sum_{\substack{j=1}}^n \tilde b_{ij}\le (\frac{1}{\lambda_0}+L)^{-1}.\label{etauu:3}
\end{align}
}
\end{lemma} 
%Similar to the parameter-dependent norms \eqref{norms} defined by using the matrix $\C$, 
Subsequently, we can use $\C_e$ to define the following parameter-dependent norms:
\begin{subequations}\label{norms:e}
\begin{align}
({\boldsymbol u},{\boldsymbol w})_{\boldsymbol U}&= ({\boldmath \epsilon}({\boldsymbol u}),{\boldmath \epsilon}({\boldsymbol w}))
+ \lambda({\text{div\,}} {\boldsymbol u},{\text{div\,}} {\boldsymbol w}),\label{norms-u:e}\\
({\boldsymbol v}, {\boldsymbol z})_{\boldsymbol V_e}&= \sum_{i=1}^n(R_i^{-1}{\boldsymbol v}_i,{\boldsymbol z}_i)
+ (\C_e^{-1}  {\text{Div\,}} {\boldsymbol v} ,{\text{Div\,}} {\boldsymbol z}),\label{norms-v:e}\\
({\boldsymbol p},{\boldsymbol q})_{\boldsymbol P_e}&= (\C_e {\boldsymbol p},{\boldsymbol q}).\label{norms-p:e}
\end{align}
\end{subequations}
As stated in the following theorem, the fixed-stress coupling iteration for the MPET system converges uniformly. 
%}
\begin{theorem}\label{theorem2}
Consider the fixed-stress coupling iteration according to Algorithm~\ref{alg1} and assume that $L\ge 1/(\lambda+c_K^2)$.
Then the errors $\be_u^m$, $\be_v^m$ and $\be_p^m$ defined in~\eqref{errors_cont_prob}, measured in the
norms induced by~\eqref{norms:e}, satisfy the estimates
\begin{equation}\label{error_eu}
\| \be_u^m\|_{\bU} \le C_u[\normalfont{\text{rate}}(\lambda)]^{m},
\end{equation}
\begin{equation}\label{error_ev_ep}
\| \be_v^m\|_{\bV_e}+ \| \be_p^m\|_{\bP_e} \le C_{vp}[\normalfont{\text{rate}}(\lambda)]^{m},
\end{equation}
where the constants $C_u$ and $C_{vp}$ are independent of the model parameters and the time step size $\tau$.
%(and
Furthermore, 
the convergence rate $\rm{rate}(\lambda)$ satisfies~\eqref{rate_gen}.
%).
\end{theorem}
\begin{proof}
In the same manner as we derived~\eqref{eq:18a} we find 
$$
\|\ep(\bm e_u^{m+1})\|^2 + \lambda \|\divv \bm e_u^{m+1}\|^2\le \left( \frac{1}{c_K^2+\lambda}\right)\left\|\displaystyle\sum_{i=1}^{n} e_{p_i}^{m+1}\right\|^2,
$$
which shows~\eqref{error_eu}.
Moreover, rewriting the error equations \eqref{error:1}--\eqref{error:3} and using the definition of~$\C_L$
we deduce the variational problem 
\begin{equation}\label{errors:2}
\begin{array}{c}
 (\ep(\bm e_u^{m+1}),\ep(\bw)) + \lambda (\divv \bm e_u^{m+1},\divv \bw)-\left(\displaystyle\sum_{i=1}^{n} e_{p_i}^{m+1},\divv \bw\right)=0,
%\left(\displaystyle\sum_{i=1}^{n} (e_{p_i}^{m}-e_{p_i}^{m+1}),\divv \bw\right),
\\ \\
 (R^{-1} \bm e_v^{m+1},\bz)- (\boldsymbol e_{p}^{m+1},{\rm Div} \bz)  =0,
\\
-\left(\divv \bm e_u^{m+1},\displaystyle\sum_{i=1}^{n} q_i\right)-( {\rm Div} \bm e_v^{m+1},\boldsymbol q) 
-((\Lambda_1+\Lambda_2+\C_L) \boldsymbol e_p^{m+1},\boldsymbol q)%}
\\\\
=
-L \left(\displaystyle\sum_{i=1}^{n} e_{p_i}^{m}, \displaystyle\sum_{i=1}^{n} q_i\right) %}
+\left(\divv \bm e_u^{m}-\divv \bm e_u^{m+1},\displaystyle\sum_{i=1}^{n} q_i\right).
\end{array}
\end{equation}
 
Denote $g_e=-L\displaystyle\sum_{i=1}^{n} e_{p_i}^{m}+\divv \bm e_u^{m}-\divv \bm e_u^{m+1}$, then by the triangle inequality,
\eqref{estimate:div} and the contraction estimate \eqref{error_sum_p}, it follows that
\begin{equation}\label{error:stability1}
\begin{split}
\|g_e\|&=\left\|-L\displaystyle\sum_{i=1}^{n} e_{p_i}^{m}+\divv \bm e_u^{m}-\divv \bm e_u^{m+1}\right \|\\
&\leq  L\left \|\displaystyle\sum_{i=1}^{n} e_{p_i}^{m}\right \|+\frac{1}{\lambda+c_K^2}\left \|\displaystyle\sum_{i=1}^{n} e_{p_i}^{m}-\displaystyle\sum_{i=1}^{n} e_{p_i}^{m+1}\right \| \\
&\leq  L\left \|\displaystyle\sum_{i=1}^{n} e_{p_i}^{m}\right \|+\frac{2}{\lambda+c_K^2}\left \|\displaystyle\sum_{i=1}^{n} e_{p_i}^{m}\right \|\\
&\leq  3L\left \|\displaystyle\sum_{i=1}^{n} e_{p_i}^{m}\right \|\leq 3L [\normalfont{\text{rate}}(\lambda)]^{m} \left \|\displaystyle\sum_{i=1}^{n} e_{p_i}^{0}\right \|.
\end{split}
\end{equation}
Next, by taking $\boldsymbol f=\boldsymbol 0$, $\boldsymbol g=(g_e,g_e,\cdots,g_e)^T$ and replacing $\C_1+\C_2$ by $\C_1+\C_2+\C_L$
in \eqref{MPET_weak} and using the uniform stability estimate~\eqref{MPET_stability_estimate} with $\C$ replaced by $\C_e$, we obtain 
\begin{equation}\label{error:stability2}
\| \be_u^{m+1}\|_{\bU}+\| \be_v^{m+1}\|_{\bV_e}+\| \be_p^{m+1}\|_{\bP_e}\leq C_1 \|\bm g\|_{\bm P_e^*}=C_1 \|\C_e^{-\frac{1}{2}} \bm g\|=C_1 (\C_e^{-1} \bm g,\bm g)^{\frac{1}{2}}.
\end{equation}
Further, by Lemma \ref{etauu} and \eqref{error:stability1}, we have 
\begin{equation}\label{error:stability3}
\begin{split}
(\C_e^{-1} \bm g,\bm g)\leq (\tilde{\C}^{-1}{\boldsymbol g},{\boldsymbol g})&=(\tilde{\C}^{-1} (\underbrace{g_e,g_e\dots,g_e}_n)^T,(\underbrace{g_e,g_e\dots,g_e}_n)^T)\\
&=\left(\sum_{i=1}^{n}\sum_{j=1}^{n}\tilde{b}_{ij}\right)(g_e,g_e)\leq \left(\frac{1}{\lambda_0}+L\right)^{-1} (g_e,g_e)\\
&\leq 9\left(\frac{1}{\lambda_0}+L\right)^{-1} L^2 [\normalfont{\text{rate}}(\lambda)]^{2m} \left\|\displaystyle\sum_{i=1}^{n} e_{p_i}^{0}\right\|^2\\
&\leq 9 L [\normalfont{\text{rate}}(\lambda)]^{2m} \left\|\displaystyle\sum_{i=1}^{n} e_{p_i}^{0}\right\|^2.
\end{split}
\end{equation}
Combining \eqref{error:stability2} and \eqref{error:stability3} then implies \eqref{error_eu} and \eqref{error_ev_ep}. 
%} 
\end{proof}

%\newpage

\section{Discrete MPET problem}\label{sec:uni_stab_disc_model} 
In this section, mass conservative discretizations of the MPET model are discussed, see also~\cite{Hong2018conservativeMPET, HongKraus2017parameter}. 
The analysis here can be similarly used for other stable discretizations of the MPET model.

%In this section, we consider the fixed-stress method of the discontinuous Galerkin (DG) discretizations for the MPET model proposed in \cite{Hong2018conservativeMPET}. 
% 

\subsection{Notation}

We consider a shape-regular triangulation $\mathcal{T}_h$ of $\Omega$ into triangles/tetrahedrons. Here, the subscript $h$ indicates the mesh-size. 
The set of all interior edges/faces and the set of all boundary edges/faces of $\mathcal{T}_h$  are denoted by
$\mathcal{E}_h^{I}$ and $\mathcal{E}_h^{B}$ respectively and their union by $\mathcal{E}_h$.
%Let $\mathcal{T}_h$ be a shape-regular triangulation of mesh-size $h$ of
%the domain $\Omega$ into triangles $\{T\}$ and define the set of all interior edges (or faces) of  $\mathcal{T}_h$ by
%$\mathcal{E}_h^{I}$ and the set of all boundary edges (or faces) by
%$\mathcal{E}_h^{B}$. Let $\mathcal{E}_h=\mathcal{E}_h^{I}\cup \mathcal{E}_h^{B}$.

We define the broken Sobolev spaces  %, we introduce the spaces
$$
H^s(\mathcal{T}_h)=\{\phi\in L^2(\Omega), \mbox{ such that } \phi|_T\in H^s(T) \mbox{ for all } T\in \mathcal{T}_h \}
$$
for $s\geq 1$. 

%{The vector functions are represented in column-wise.}

%As we consider discontinuous Galerkin (DG) discretizations, we also define some trace operators. 
We next introduce the notion of jumps $[\cdot]$ and averages $\{\cdot \}$. %for the functions $q\in H^1(\mathcal{T}_h)$ and $\bm v \in H^1(\mathcal{T}_h)^d$.
Let $T_1$ and $T_2$ be two elements from the triangulation sharing an edge or face $e$ and $\bm n_1$ and $\bm n_2$ be the corresponding unit normal vectors to $e$ 
pointing to the exterior of $T_1$ and $T_2$. 
Then for $q\in H^1(\mathcal{T}_h)$, $\bm v \in H^1(\mathcal{T}_h)^d$ and $\bm \tau \in H^1(\mathcal{T}_h)^{d\times d}$ and any $e\in \mathcal{E}_h^{I}$ we define
%and the jumps are given by
\begin{equation*}
[q]=q|_{\partial T_1\cap e}-q|_{\partial T_2\cap e},\quad
[\bm v]=\bm v|_{\partial T_1\cap e}-\bm v|_{\partial T_2\cap e}
\end{equation*}
and
%, 
%the averages are defined as
\begin{equation*}
\begin{split}
\{\bm v\} &=\frac{1}{2}(\bm v|_{\partial T_1\cap e}\cdot \bm n_1-\bm
v|_{\partial T_2\cap e}\cdot \bm n_2), \quad 
\{\bm \tau\}=\frac{1}{2}(\bm \tau|_{\partial T_1\cap
e} \bm n_1-\bm \tau|_{\partial T_2\cap e} \bm n_2),
\end{split}
\end{equation*}
while for 
%When 
$e \in  \mathcal{E}_h^{B}$, %then the above quantities are defined as
\[[q]=q|_{e}, ~~ [\bm v]=\bm v|_{e},\quad
\{\bm v\}=\bm v |_{e}\cdot \bm n,\quad
\{\bm \tau\}=\bm \tau|_{e}\bm n.
\]

\subsection{Mixed finite element spaces and discrete formulation}\label{subsec:DGdiscretization}
In order to discretize the flow equations, we use a mixed finite element method to approximate the fluxes and pressures 
whereas for the mechanics problem we apply a discontinuous Galerkin method to approximate the displacement. 
%field. 
The considered finite element spaces are denoted by:
\begin{eqnarray*}
\bm U_h&=&\{\bm u \in H(\div ;\Omega):\bm u|_T \in \bm U(T),~T \in \mathcal{T}_h;~ \bm u \cdot
\bm n=0~\hbox{on}~\partial \Omega\},
\\[1ex]
\bm V_{i,h}&=&\{\bm v \in H(\div ;\Omega):\bm v|_T \in \bm V_i(T),~T \in \mathcal{T}_h;~ \bm v \cdot
\bm n=0~\hbox{on}~\partial \Omega\},~~i=1,\dots,n,
\\
P_{i,h}&=&\{q \in L^2(\Omega):q|_T \in Q_i(T),~T \in \mathcal{T}_h; ~\int_{\Omega}q dx=0\},~~i=1,\dots,n,
\end{eqnarray*}
where $\bm V_{i}(T)/Q_{i}(T)={\rm RT}_{l-1}(T)/{\rm P}_{l-1}(T)$ and $\bm U(T) = {\rm BDM}_l(T)$ or $\bm U(T) = {\rm BDFM}_l(T)$ for $l\ge 1$.
%The discretizations we analyze in the present context define the local spaces $\bm U(T)/\bm V_i(T)/Q_i(T)$
%via the triplets $BDM_l(T)/$ $RT_{l-1}(T)/P_{l-1}(T)$, or
%$BDFM_l(T)/RT_{l-1}(T)/P_{l-1}(T)$ for $l\ge 1$. 
For each of these choices, we would like to point out that $\div  \bm U(T)=\div  \bm V_i(T)=Q_i(T)$ is satisfied.
%
%We recall the following basic approximation properties of these spaces: For all
%$K\in \mathcal{T}_h$ and for all $\bm u \in H^s(K)^d$, there exists $\bm u_I\in
%\bm U(K)$ such that
%\begin{equation}\label{eq:70}
%\|\bm u-\bm u_I\|_{0,K}+h_K|\bm u-\bm u_I|_{1,K}+h_K^2|\bm u-\bm
%u_I|_{2,K}
%\leq C h_K^s|\bm u|_{s,K}, ~2\le s \leq l+1.
%\end{equation}

%Note that the normal component of any $\bm u\in \bm U_h$ is continuous
%on the internal edges and vanishes on the boundary edges. 
As commented also in~\cite{HongKraus2017parameter,Hong2018conservativeMPET}, for all $e\in \mathcal{E}_h$ and for all $\bm \tau\in H^1(\mathcal{T}_h)^d, \bm u\in \bm U_h$ it holds that
\begin{equation}\label{eq:72}
\int_e[\bm u_n]\cdot \bm \tau ds=0, \quad \mbox{from which it follows} \quad
\int_e[\bm u]\cdot \bm \tau ds=\int_e[\bm u_t]\cdot \bm \tau ds,
\end{equation}
where $\bm u_n$ and $\bm u_t$ denote the normal and tangential component of $\bm u$ respectively.

Using the notation
$$
\bm v_h^T=(\boldsymbol v^T_{1,h}, \cdots \boldsymbol v^T_{n,h}), \quad \bm p_h^T=(p_{1,h},\cdots, p_{n,h}), 
\quad \bm z_h^T=(\boldsymbol z^T_{1,h}, \cdots \boldsymbol z^T_{n,h}),\quad \bm q_h^T=(q_{1,h},\cdots, q_{n,h}),
$$
$$
\bm V_h=\boldsymbol V_{1,h}\times\cdots\times\boldsymbol V_{n,h},
\quad \bm P_h= P_{1,h}\times\cdots\times P_{n,h}, \quad  \bm X_h = \bm U_h\times\bm V_h\times \bm P_h
$$
the discretization of the variational problem~\eqref{MPET_weak} can be expressed as:
Find $(\boldsymbol u_h;\bm v_h;\bm p_h, )\in  \bm X_h$, such that for
any $(\boldsymbol w_h; \bm z_{h};\bm q_{h})\in  \bm X_{h}$
and $i=1,\dots, n$
\begin{subequations}\label{eq:75-77}
\begin{eqnarray}
 a_h(\bm u_h,\bm w_h) +\lambda ( \div \boldsymbol  u_h, \div \boldsymbol  w_h)
 - \sum_{i=1}^n(p_{i,h}, \div \boldsymbol  w_h)&=&(\boldsymbol f, \boldsymbol w_h),\label{eq:75}\\
  (R^{-1}_i\bv_{i,h},\bz_{i,h}) {-} (p_{i,h},\divv \bz_{i,h}) &=& 0,  \label{eq:76} \\
 	 -(\divv\bu_h,q_{i,h}) - (\divv\bv_{i,h},q_{i.h})  +\tilde{\alpha}_{ii} (p_{i,h},q_{i,h})
 	 +\sum_{\substack{{j=1}\\j\neq i}}^{n}\alpha_{ij} (p_{j,h},q_{i,h})&=& (g_i,q_{i,h})
	 %, 
	 \label{eq:77}
\end{eqnarray}
\end{subequations}
where
\begin{eqnarray}
a_h(\bm u,\bm w)&=&\label{78}
\sum _{T \in \mathcal{T}_h} \int_T \ep(\bm{u}) :
\ep(\bm{w}) dx-\sum_{e \in \mathcal{E}_h} \int_e \{\ep(\bm{u})\} \cdot [\bm w_t] ds\\
&&\nonumber-\sum _{e \in \mathcal{E}_h} \int_e \{\ep(\bm{w})\} \cdot [\bm u_t]ds+\sum _{e
\in \mathcal{E}_h} \int_e \eta h_e^{-1}[ \bm u_t] \cdot [\bm w_t] ds,
\end{eqnarray}
$\tilde{\alpha}_{ii}=-\alpha_{p_i}-\alpha_{ii}$, and $\eta $ is a stabilization parameter independent of the
parameters $\lambda,\,R_i^{-1}$, $\alpha_{p_i}$,  ${\alpha}_{ij}$, where $i,j \in \{1,\dots,n\}$, the network scale $n$
and the mesh size $h$. 

The discrete variational problem~\eqref{eq:75-77} corresponds to the weak formulation \eqref{MPET_weak} with
homogeneous boundary conditions. The DG discretizations for general rescaled boundary conditions can be found
in~\cite{Hong2018conservativeMPET, HongKraus2017parameter}.
%\begin{remark}
%The general rescaled boundary conditions 
%\begin{subequations}\label{eq:Biot_BC}
%\begin{eqnarray}
%p_i&=& p_{i,D} \qquad \mbox{on }~~ \Gamma_{p_i,D},~~ i=1,\dots,n,\\ 
%\bm v_i\cdot {\bm n} &=& q_{i,N}  \qquad \mbox{on}~~~ \Gamma_{p_i,N}, ~~i=1,\dots,n,\\ 
%\u&=& {\bu}_D  \quad ~~~~\mbox{on}~~~\Gamma_{\vek{u},D}, \\ 
%{(\bm \sigma -\sum_{i=1}^np_i\boldsymbol I){\vek{n}} }&=& {\bm g}_N \qquad~\,\mbox{on}~~~\Gamma_{\vek{u},N}.
%\end{eqnarray}
%\end{subequations}
%can be incorporated as explained in~\cite{HongKraus2017parameter}. 
%\end{remark}
%
%\begin{proposition}\label{prop:mass_cons}
%Let $(\boldsymbol u_h; \boldsymbol v_h; \bm p_h)\in \boldsymbol U_h\times\boldsymbol V_h\times \bm P_h$ be the solution
%of \eqref{eq:75}-\eqref{eq:77}, then the pointwise mass conservation equation is satisfied, that is
%\begin{equation}\label{eq:conservative}
%-{\rm div} \bu_h - {\rm div}\bv_{i,h}  -(\alpha_{p_i} +\alpha_{ii})p_{i,h} + \sum_{\substack{{j=1}\\j\neq i}}^{n}{\alpha}_{ij} p_{j,h} = Q_{i,h}g_i,~i=1,\dots,n ,~\forall x\in K, \forall K\in  \mathcal{T}_h,
%\end{equation}
%where the $L^2$-projection on $P_{i,h}$ is denoted by $Q_{i,h}$.
%%
%Hence, if $g_i=0$ then $-{\rm div} \bu_h - {\rm div}\bv_{i,h}  -(\alpha_{p_i} +\alpha_{ii})p_{i,h} + \sum\limits_{\substack{{j=1}\\j\neq i}}^{n}{\alpha}_{ij} p_{j,h}=0$.
%\end{proposition}

\subsection{Stability properties}
%For 
%$\bm u\in  H^1(\mathcal{T}_h)^d$,
Let $\bm u$ be a function from $\boldsymbol U_h$ and consider
%We introduce 
the mesh dependent norms
\begin{eqnarray*}
\|\bm{u}\|_h^2&=&\sum _{K \in \mathcal{T}_h} 
\|\ep(\bm{u})\|_{0,K}^2+\sum _{e \in \mathcal{E}_h} h_e^{-1}\|[ \bm u_t]\|_{0,e}^2, \\
\|\bm u\|_{1,h}^2&=&\sum _{K \in \mathcal{T}_h} \|\nabla\bm{u}\|_{0,K}^2+\sum _{e \in \mathcal{E}_h} h_e^{-1}\|[ \bm{u}_t]\|_{0,e}^2,
\end{eqnarray*}
%Next, for $\bm u\in  \boldsymbol U_h$, we define 
%the ``DG''-norm
\begin{equation}\label{DGnorm}
\|\bm u\|^2_{DG}=\sum _{K \in \mathcal{T}_h} \|\nabla\bm{u}\|_{0,K}^2+\sum _{e \in \mathcal{E}_h} h_e^{-1}\|[ \bm{u}_t]\|_{0,e}^2+\sum _{K \in \mathcal{T}_h}h_K^2|\bm{u}|^2_{2,K}
\end{equation}
and %, finally, the mesh-dependent norm $\|\bm \cdot \|_{\bm U_h}$
\begin{equation}\label{U_hnorm}
\|\bm u\|^2_{\bm U_h}=\|\bm u\|^2_{DG}+\lambda \|\div \bm u\|^2.
\end{equation}

The well-posedness and approximation properties of the DG formulation are detailed in~\cite{hong2016robust, honguniformly}.  
Here we briefly present some important results:
\begin{itemize}
\item $\|\cdot\|_{DG}$, $\|\cdot\|_h$, and $\|\cdot\|_{1,h}$ are equivalent on $\bm U_h$;
%, 
that is
$$
\|\bm{u}\|_{DG}\eqsim  \|\bm{u}\|_h\eqsim\|\bm u\|_{1,h},\,\mbox{for all}~\bm u \in \bm U_h.
$$
\item
%Secondly, the bilinear form 
$a_h(\cdot,\cdot)$ from~\eqref{78} is continuous and it holds true that
\begin{eqnarray}\label{continuity:a_h}
|a_h(\bm u,\bm w)|&\lesssim& \| \bm u  \|_{DG}  \| \bm w  \|_{DG},\quad\mbox{for all}\quad \bm u,~\bm w\in H^2(\mathcal{T}_h)^d.
\end{eqnarray}
%Thirdly, the LBB conditions 
\item The inf-sup conditions
\begin{equation}\label{inf-sup}
\begin{split}
 & \inf_{(q_{1,h},\cdots,q_{n,h})\in P_{1,h} \times \cdots \times P_{n,h}}\sup_{\bm{u}_h\in \bm{U}_h}\frac{(\operatorname{div}\bm{u}_h,\sum\limits_{i=1}^n q_{i,h})}{\|\bm{u}_h\|_{1,h}\|\sum\limits_{i=1}^n q_{i,h}\|}\geq \beta_{sd},\\
 &  \inf_{q_{i,h}\in P_{i,h}}\sup_{\bm{v}_{i,h}\in \bm{V}_{i,h}}\frac{(\operatorname{div}\bm{v}_{i,h},q_{i,h})}{\|\bm{v}_{i,h}\|_{\div}\|q_{i,h}\|}\geq \beta_{dd}, \quad i=1,\dots,n,
\end{split}
\end{equation}
are valid for our choice of {$\bm U_h, \bm{V}_h$ and $\bm P_h$}, see~\cite{schtzau2002mixed}, 
and the positive constants $\beta_{sd}$ and $\beta_{dd}$ are independent of %the parameters 
$\lambda$, $R_i^{-1}$, $\alpha_{p_i}$, ${\alpha}_{ij}$ for $i,j \in \{1,\dots,n\}$, the network scale $n$ and the mesh size $h$. 
\item %Finally, using standard techniques one can show that 
$a_h(\cdot,\cdot)$ is coercive, namely 
\begin{equation}\label{coercivity:a_h}
a_h(\bm{u}_h,\bm{u}_h)\geq \alpha_a \|\bm{u}_h\|^2_h,\quad\mbox{for all}\quad~\bm{u}_h\in\bm{U}_h,
\end{equation}
where $\alpha_a>0$ is a constant independent of the model and discretization parameters $\lambda,R_i^{-1}, \alpha_{p_i}$, ${\alpha}_{ij}, i,j \in \{1,\dots,n\}$, $n$ and $h$.
\end{itemize}

Using the definition of the matrices $\C_1$ and $\C_2$, we define the bilinear form 
\begin{equation}\label{eq:79}
\begin{split}
&\mathcal A_h((\bu_h; \bv_{h}; \bm p_{h}),(\bw_h; \bz_{h}; \bm q_{h}))=a_h(\bm u_h,\bm w_h)+\lambda ( \div \boldsymbol  u_h, \div \boldsymbol  w_h) -\sum_{i=1}^{n} (p_{i,h},{\rm div} \bw_h) 
\\& +\sum_{i=1}^{n}(R^{-1}_i\bv_{i,h},\bz_{i,h})- (\bp_h, \Divv \bz_h)
-(\divv\bu_h,\sum_{i=1}^{n}q_{i,h})- (\Divv \bv_h,\bq_h) - ((\C_1+\C_2)\bp_h,\bq_h)
\end{split}
\end{equation}
related to problem~\eqref{eq:75}--\eqref{eq:77}.

Similar to Theorem~\ref{continu_stability}, the following uniform stability results can be found in \cite{Hong2018conservativeMPET}.
\begin{theorem}\label{boundd}
~
\begin{itemize}
\item[(i)] For any $(\bm u_h; \bm v_h;\bm p_h)\in \bm X_h, (\bm w_h; \bm z_h; \bm q_h)\in \bm X_h$
there exists a positive constant $C_{bd}$ independent of  the parameters $\lambda$, $R_i^{-1}$, $\alpha_{p_i}$, ${\alpha}_{ij}$, 
$i,j \in \{1,\dots,n\}$, the network scale~$n$ and the mesh size~$h$ such that the inequality 
\begin{equation*}
|\mathcal A_h((\bu_h;\bv_{h};\bm p_{h}),(\bw_h;\bz_{h};\bm q_{h}))|\le C_{bd} (\|\bm u_h\|_{\bm U_h}+\|\bm v_h\|_{\bm V}
+\|\bm p_h\|_{\bm P})  (\|\bm w_h\|_{\bm U_h}+\|\bm z_h\|_{\bm V}+\|\bm q_h\|_{\bm P})
\end{equation*}
holds true.
\item[(ii)]  There exists a constant $\beta_0>0$ independent of  the model and discretization parameters $\lambda$, $R_i^{-1}$, $\alpha_{p_i}$,
${\alpha}_{ij}$, $i,j \in \{1,\dots,n\}$, $n$ and~$h$, such that 
%{\small
\begin{equation}\label{eq:80}
\displaystyle\inf_{(\boldsymbol u_h;  \boldsymbol v_h; \bm p_h)\in  \bm X_h} 
\sup_{(\boldsymbol w_h;\boldsymbol z_h; \bm q_h)\in \boldsymbol X_h}\frac{\mathcal A_h((\bu_h;\bv_{h};\bm p_{h}),(\bw_h;\bz_{h};\bm q_{h}))}{(\|\boldsymbol u_h\|_{\bm U_h}+
\|\boldsymbol v_h\|_{\boldsymbol V}+\|\bm p_h\|_{\bm P})(\|\boldsymbol w_h\|_{\bm U_h}+\|\boldsymbol z_h\|_{\boldsymbol V}+\|\bm q_h\|_{\bm P})}\geq \beta_0.
\end{equation}
%}

\item[(iii)] Let $(\boldsymbol u_h;  \boldsymbol v_h; \bm p_h)\in \bm X_h$ solve~\eqref{eq:75}-\eqref{eq:77}
and
$$\|\boldsymbol f\|_{\boldsymbol U_h^*}=
\sup\limits_{\boldsymbol w_h\in \boldsymbol U_h}\frac{(\boldsymbol f, \boldsymbol w_h)}
{\|\boldsymbol w_h\|_{\boldsymbol U_h}},\quad \|\bm g\|_{\bm P^*}=
\sup\limits_{\bm q_h\in \bm P_h}\frac{(\bm g, \bm q_h)}{\|\bm q_h\|_{\bm P}}.$$
Then the estimate
\begin{equation}
\|\boldsymbol u_h\|_{\boldsymbol U_h}+\|\boldsymbol  v_h\|_{\boldsymbol V}+\|\bm p_h\|_{\bm P}\leq C_2 (\|\boldsymbol f\|_{\boldsymbol U_h^*}+\|\bm g\|_{\bm P^*})
\end{equation} 
holds with 
a constant $C_2$ independent of $\lambda$, $R_i^{-1}$, $\alpha_{p_i}$, ${\alpha}_{ij}$, $i,j \in \{1,\dots,n\}$, the network scale $n$ and the mesh size $h$.
\end{itemize}
\end{theorem}

\section{Fixed-stress method for the discrete MPET model}
In the manner of Algorithm~\ref{alg1}, we formulate the fixed-stress method for the mixed continuous-discontinuous Galerkin finite element method~\eqref{eq:75-77}:
%.
\begin{algorithm}[H]
\textbf{Step a}: Given $\bu_h^{m}$, we solve for $\bv_{i,h}^{m+1}$ and $p_{i,h}^{m+1}$%, $1 \le i \le n$,
\begin{equation}\label{fixed_stress_eq1:dis}
\begin{array}{l}
( - \divv\bv_{i,h}^{m+1},q_{i,h}) -((\alpha_{p_{i,h}} + {\alpha}_{ii}) p_{i,h}^{m+1},q_{i,h})+\left(\displaystyle\sum_{\substack{{j=1}\\j\neq i}}^{n}{\alpha}_{ij} p_{j,h}^{m+1},q_{i,h}\right) -L\left( \displaystyle\sum_{j=1}^{n} p_{j,h}^{m+1},q_{i,h}\right)
\\[1ex] \hspace{2.8cm}= (g_i,q_{i,h})-L \left(\displaystyle\sum_{j=1}^{n}p_{j,h}^{m},q_{i,h}\right)+(\divv \bu^{m},q_{i,h}),\quad 1 \le i \le n,
\end{array}
\end{equation}
and 
\begin{equation}\label{fixed_stress_eq2:dis}
(R_i^{-1} \bv_{i,h}^{m+1},\bz_{i,h})- (p_{i,h}^{m+1},\divv \bz_{i,h}) =0,\quad 1 \le i \le n.
\end{equation}
\\ \\
\noindent
\textbf{Step b}: Given $\bv_{i,h}^{m+1}$ and $p_{i,h}^{m+1}$, we solve for $\bu_h^{m+1}$ 
\begin{equation}\label{fixed_stress_eq3}
a_h(\bu_h^{m+1},\bw_h) + \lambda (\divv \bu_h^{m+1},\divv \bw_h)  = (\bff,\bw_h) +\sum_{i=1}^{n} ( p_{i,h}^{m+1},\divv \bw_h).
\end{equation}
\caption{\textbf{: Fixed-stress method for the discrete MPET problem}}\label{alg:dis}
\end{algorithm}
The main convergence result for Algorithm~\ref{alg:dis} is formulated in terms of the following quantities corresponding to the discrete case:
\begin{subequations}\label{errors_cont_prob:dis}
\begin{eqnarray}
\be_{u_h}^k & = & \bm u_h^k-\bm u_h\in \bU_h, \\
\be_{v_{i,h}}^k&=&\bm v_{i,h}^k-\bm v_{i,h} \in \bV_{i,h}, \quad i=1,\ldots,n, \\
e_{p_{i,h}}^k&=&p_{i,h}^k-p_{i,h} \in P_{i,h}, \quad i=1,\ldots,n, 
\end{eqnarray}
\end{subequations}
denoting the errors of the $k$-th iterates $\bm u_h^k$, $\bm v^k_{i,h}$, $p^k_{i,h}$, $i=1,\ldots,n$ generated by Algrorithm~\ref{alg:dis}. 
In the discrete case, the useful constant for defining the tuning parameter $L$ is the constant $c_{K_d}$ from the estimate
\begin{equation}\label{ck:dis}
a_h(\bw_h, \bw_h)\ge c^2_{K_d}\|\divv\bw_h\|^2 \quad \text{for all }\bw_h \in \bU_h.
\end{equation}
Note that $c_{K_d}$ is strictly positive and independent of the mesh size $h$. 
%which will be used for $\bw_h =\bm e_{u_h}^{m+1}-\bm e_{u_h}^{m}$ in the proof of the next Lemma.

Using the approach applied to proving Lemma \ref{lemma1}, for the continuous MPET model we obtain the corresponding lemma for the discrete case as follows:
\begin{lemma}\label{lemma1:dis}
The errors $\be_{u_h}^{m+1}$, $\be_{v_h}^{m+1}$ and $\be_{p_h}^{m+1}$ of the $(m+1)$-st fixed-stress iterate generated by Algorithm~\ref{alg:dis} for
$L \ge \displaystyle\frac{1}{\lambda+c_{K_d}^2}$ satisfy the estimate
\begin{equation}\label{eq:lemma1:dis}
\begin{array}{l}
\displaystyle \frac{1}{2}\big(a_h(\bm e_{u_h}^{m+1},\bm e_{u_h}^{m+1})+ \lambda \| \normalfont{\divv} \bm e_{u_h}^{m+1}\|^2\Big)+\|R^{-1/2}\bm e_{v_h}^{m+1}\|^2+\|(\Lambda_1+\Lambda_2)^{1/2}\boldsymbol e_{p_h}^{m+1}\|^2
 \\
%\hspace{0.8cm}
\qquad +\displaystyle \frac{L}{2} \left\| \displaystyle\sum_{i=1}^{n} e_{p_{i,h}}^{m+1}\right\|^2
\leq  \frac{L}{2}\left\|\sum_{i=1}^{n} e_{p_{i,h}}^{m}\right\|^2, \quad m=0,1,2,\ldots.
\end{array}
\end{equation}
%for $m=0,1,2,\ldots.$
\end{lemma}
By Lemma \ref{lemma1:dis}, again following the proof of Theorem \ref{theorem1} for the continuous MPET model, we obtain the corresponding statements, Theorem~\ref{theorem1:dis}, 
for the discrete case:
%.
\begin{theorem}\label{theorem1:dis}
Let $c_{K_d}$ and $\beta_{sd}$ denote the constants in~\eqref{ck:dis} and~\eqref{inf-sup} respectively.
The single rate fixed-stress iterative method for the discrete static MPET problem~\eqref{eq:75-77} defined in Algorithm~\ref{alg:dis} is a
contraction that converges linearly for any $L\ge 1/(\lambda+c_{K_d}^2)$ independent of the model parameters, the time step
size $\tau$ and the mesh size $h$. 
The errors $\be_{p_h}^m$ in this case satisfy the inequality
\begin{equation}\label{error_sum_p:dis}
\left\| \sum_{i=1}^n e_{p_{i,h}}^{m+1} \right\|^2 \le \normalfont{\text{rate}_d}^2(\lambda) \left\| \sum_{i=1}^n e_{p_{i,h}}^m \right\|^2
\end{equation}
where
\begin{equation}\label{rate_gen:dis}
\normalfont{\text{rate}_d}^2(\lambda) \le \displaystyle\frac{1}{\frac{L^{-1}}{\beta_{sd}^{-2}+\lambda}+1}.
\end{equation}
For $L =\displaystyle \frac{1}{\lambda+c^2_{K_d}}$, the convergence factor in~\eqref{error_sum_p:dis}
can be estimated by
\begin{equation}\label{rate_spec:dis}
\normalfont{\text{rate}_d}^2(\lambda)\le \displaystyle \frac{1}{\frac{\lambda+c_{K_d}^2}{\beta_{sd}^{-2}+\lambda}+1}
\le \max\left\{ \displaystyle\frac{\beta_{sd}^{-2}}{c_{K_d}^2+\beta_{sd}^{-2}},\displaystyle\frac{1}{2}\right\}.
\end{equation}
\end{theorem}
Note that Theorem \ref{theorem1:dis} only gives the convergence rate of $\be_{p_h}^m$ in the semi-norm $\left\| \sum_{i=1}^n e_{p_{i,h}}^{m} \right\|$. However, we can combine 
the estimates in Theorem \ref{theorem1:dis} with the uniform stability result presented in Theorem \ref{boundd} and follow the proof of Theorem \ref{theorem2} to obtain the following 
convergence results for $\be_{u_h}^m, \be_{v_h}^m$ and $\be_{p_h}^m$ in their respective parameter-dependent full norms:
\begin{theorem}\label{theorem2h}
The errors $\be_{u_h}^m$, $\be_{v_h}^m$ and $\be_{p_h}^m$ defined in~\eqref{errors_cont_prob:dis} measured in the norms induced by~\eqref{U_hnorm} and \eqref{norms}
satisfy the estimates:
\begin{equation}\label{error_eu:dis}
\| \be_{u_h}^m\|_{\bU_h}^2 \le C_{ud}[\normalfont{\text{rate}_d}(\lambda)]^{2m},
\end{equation}
\begin{equation}\label{error_ev_ep:dis}
\| \be_{v_h}^m\|_{\bV_e}^2 + \| \be_{p_h}^m\|_{\bP_e}^2 \le C_{vpd}[\normalfont{\text{rate}_d}(\lambda)]^{2m}, %},
\end{equation}
where the constants $C_{ud}$ and $C_{vpd}$ are independent of the model parameters, the time step size and the mesh size.
\end{theorem}

\section{Numerical results}\label{numerics}

In our numerical test setup, we assume that:

 %The general setting of the numerical tests is as in~\cite{Hong2018conservativeMPET}:
 \begin{itemize} 
 \item $\Omega=[0,1]$ is partitioned into $2N^2$ right-angled triangles with catheti of length $h = \displaystyle\frac{1}{N};$ 
 \item Problem~\eqref{MPET_weak} is discretized by a strongly conservative discontinuous Galerkin method based on a mixed finite element space formed by 
 the triplet of ${\rm BDM}_1/{\rm RT}_0/{\rm P}_0^{dc}$ elements;
 \item the constant for the fixed-stress splitting is $L=\displaystyle\frac{1}{1+\lambda}$; 
 \item the iterative process is terminated when residual reduction by a factor $10^8$ in the combined norm induced by the inner products~\eqref{norms} (the norm induced by the inverse of the preconditioner) 
is reached.
 \end{itemize}
Numerical experiments have been performed in FEniCS, \cite{AlnaesBlechta2015a,LoggMardalEtAl2012a}, and their aim was:
\begin{itemize}
\item[(i)] to validate the theoretical estimates for the convergence of the fixed-stress splitting;
\item[(ii)] to compare the performance of the latter with the preconditioned MinRes algorithm using the norm-equivalent preconditioner proposed in~\cite{Hong2018conservativeMPET}. 

\end{itemize}

\subsection{The two-network model}

The Biot-Barenblatt model involves two pressures and two fluxes. In our notation, it has the following formulation:

\begin{subequations}
\begin{align}
-{\rm div} (\bm \sigma-p_1\bm I-p_2\bm I ) &= \bm f ,\\
R^{-1}_i\bm v_i + \nabla p_i &= 0, \ \quad i=1,2, \\
-{\rm div} \bm u - {\rm div}\bm v_i -\alpha_{p_i}p_i + \sum_{\substack{j=1\\j\neq i}}^2\alpha_{ij} p_j &= g_i, \quad i=1,2 .
\end{align}
\label{eq:beweak}
\end{subequations}

%\noindent
Specifically, the subject of numerical study in this subsection is the cantilever bracket benchmark problem, see~\cite{NAFEMS1990}, for which 
$\boldsymbol f  =  \bm 0 $, $g_1  = g_2 =  0$. The boundary $\Gamma$ of the domain $\Omega=[0,1]^2$ is split into 
bottom, right, top and left boundaries denoted by
$\Gamma_1$, $\Gamma_2$, $\Gamma_3$ and $\Gamma_4$ respectively and 
$$
\begin{array}{rcll}
(\boldsymbol {\sigma} -p_1\boldsymbol I-p_2\boldsymbol I)\boldsymbol n & = &(0,0)^T & \text{ on }\Gamma_1\cup\Gamma_2, \\
(\boldsymbol {\sigma} -p_1\boldsymbol I-p_2\boldsymbol I)\boldsymbol n& = &(0,-1)^T &\text{ on }\Gamma_3, \\
\boldsymbol u & = & \boldsymbol  0 & \text{ on } \Gamma_4, \\
p_1 & = & 2 & \text{ on } \Gamma, \\
p_2 & = & 20 & \text{ on }\Gamma. 
\end{array}
$$
 
Table~\ref{parameters_Barenblatt} gives the base values of the model parameters as taken from~\cite{Kolesov2017}. 
We have varied the parameter $K_2$ over a wider range than $K_1$ 
since, at least, for the MinRes iteration it happened to be the more interesting case.
The results in Tables~\ref{tab:2}--\ref{tab:4} show very clearly the robust behaviour of the fixed-stress iteration with respect to mesh refinements and variation of the hydraulic conductivities 
$K_1$ and $K_2$, and also $\lambda$. Furthermore, they demonstrate its advantage over the MinRes method in terms of rate of convergence.

\begin{table}[h!]
\caption{Base values of model parameters for a Barenblatt model.}
\label{parameters_Barenblatt}
\begin{center}
\begin{tabular}{ccc}
\hline 
\cellcolor{viol!35} 
parameter & 
\cellcolor{viol!20} 
value & 
\cellcolor{viol!5} 
unit \\ \hline
$\lambda$ & $4.2$ & MPa \\
$\mu$ & $2.4$ & MPa \\
$c_{p_1}$ & $54$ & (GPa)$^{-1}$ \\
$c_{p_2}$ & $14$ & (GPa)$^{-1}$\\
$\alpha_{1}$ & $0.95$ & \\
$\alpha_{2}$ & $0.12$ & \\
\multirow{2}{*}{$\beta$} & $5$ & $10^{-10}$kg/(m$\cdot$s) \\
& $100$ & $10^{-10}$kg/(m$\cdot$s) \\
$K_1$ & $6.18$ & $10^{-15}$m$^2$ \\
$K_2$ & $27.2$ & $10^{-15}$m$^2$
\end{tabular}
\end{center}
\end{table}

\begin{table}[hbtp!]
\caption{Number of preconditioned MinRes and fixed-stress splitting iterations for residual reduction by a factor $10^8$ in
the norm induced by the preconditioner when solving the Barenblatt problem.}
\label{tab:2}
%\begin{center}
\centering
\begin{tabular}{c|c||c|rr|rr|rr|rr}
$h$ & $\beta$ &   & \multicolumn{2}{c}{$K_2$} &  \multicolumn{2}{c}{$K_2\cdot 10^{2}$} &  \multicolumn{2}{c}{$K_2\cdot 10^{4}$}  &  \multicolumn{2}{c}{$K_2\cdot 10^{6}$} \\ \hline \hline
\multirow{6}{*}{$\displaystyle \frac{1}{16}$} & 
\multirow{3}{*}{5E--10}  & $K_1\cdot 10^{-2}$ & 16 & \color{orange}{8} & 21 &\color{orange}{8} & 37 & \color{orange}{8} & 29 & \color{orange}{8}  \\
                               & & $K_1\cdot 10^{-1}$ & 16 & \color{orange}{8} & 21 &\color{orange}{8} & 37 & \color{orange}{8} & 29 & \color{orange}{8}  \\
                               & & $K_1$ & 16 & \color{orange}{8} & 21 &\color{orange}{8} & 37 & \color{orange}{8} & 29 & \color{orange}{8}  \\ \cline{2-11}
& \multirow{3}{*}{1E-8}  & $K_1\cdot 10^{-2}$ & 16 & \color{orange}{8} & 21 &\color{orange}{8} & 37 & \color{orange}{8} & 29 & \color{orange}{8}  \\   
                                     & & $K_1\cdot 10^{-1}$ & 16 & \color{orange}{8} & 21 &\color{orange}{8} & 37 & \color{orange}{8} & 29 & \color{orange}{8}  \\  
                                    & & $K_1$ & 16 & \color{orange}{8} & 21 &\color{orange}{8} & 37 & \color{orange}{8} & 29 & \color{orange}{8}  \\ \hline   
                                    \multirow{6}{*}{$\displaystyle \frac{1}{32}$} & 
\multirow{3}{*}{5E--10}  & $K_1\cdot 10^{-2}$ & 16 & \color{orange}{8} & 26 &\color{orange}{8} & 38 & \color{orange}{8} & 27 & \color{orange}{8}  \\
                               & & $K_1\cdot 10^{-1}$ & 16 & \color{orange}{8} & 26 &\color{orange}{8} & 38 & \color{orange}{8} & 27 & \color{orange}{8}  \\
                               & & $K_1$ & 16 & \color{orange}{8} & 26 &\color{orange}{8} & 38 & \color{orange}{8} & 27 & \color{orange}{8}  \\ \cline{2-11}
& \multirow{3}{*}{1E-8}  & $K_1\cdot 10^{-2}$ & 16 & \color{orange}{8} & 26 &\color{orange}{8} & 38 & \color{orange}{8} & 27 & \color{orange}{8}  \\   
                                     & & $K_1\cdot 10^{-1}$ & 16 & \color{orange}{8} & 26 &\color{orange}{8} & 38 & \color{orange}{8} & 27 & \color{orange}{8}  \\  
                                    & & $K_1$ & 16 & \color{orange}{8} & 26 &\color{orange}{8} & 38 & \color{orange}{8} & 27 & \color{orange}{8}  \\ \hline    
 \multirow{6}{*}{$\displaystyle \frac{1}{64}$} & 
\multirow{3}{*}{5E--10}  & $K_1\cdot 10^{-2}$ & 18 & \color{orange}{8} & 32 &\color{orange}{8} & 38 & \color{orange}{8} & 27 & \color{orange}{8}  \\
                               & & $K_1\cdot 10^{-1}$ & 18 & \color{orange}{8} & 32 &\color{orange}{8} & 38 & \color{orange}{8} & 27 & \color{orange}{8}  \\
                               & & $K_1$ & 18 & \color{orange}{8} & 32 &\color{orange}{8} & 38 & \color{orange}{8} & 27 & \color{orange}{8}  \\ \cline{2-11}
& \multirow{3}{*}{1E-8}  & $K_1\cdot 10^{-2}$ & 18 & \color{orange}{8} & 32 &\color{orange}{8} & 38 & \color{orange}{8} & 27 & \color{orange}{8}  \\   
                                     & & $K_1\cdot 10^{-1}$ & 18 & \color{orange}{8} & 32 &\color{orange}{8} & 38 & \color{orange}{8} & 27 & \color{orange}{8}  \\  
                                    & & $K_1$ & 18 & \color{orange}{8} & 32 &\color{orange}{8} & 38 & \color{orange}{8} & 27 & \color{orange}{8}  \\ \hline                                                              
\end{tabular}
%\end{center}
\end{table}

\begin{table}[hbtp!]
%\caption{Preconditioned MinRes and {\color{orange}fixed-stress splitting} convergence history for the Barenblatt problem
%{where we have redefined $\lambda := 0.01\cdot\lambda$}.}
\caption{Number of preconditioned MinRes and fixed-stress splitting iterations for residual reduction by a factor $10^8$ in
the norm induced by the preconditioner when solving the Barenblatt problem where we have redefined $\lambda:=0.01\cdot \lambda$.}
\label{tab:3}
%\begin{center}
\centering
\begin{tabular}{c|c||c|rr|rr|rr|rr}
$h$ & $\beta$ &   & \multicolumn{2}{c}{$K_2$} &  \multicolumn{2}{c}{$K_2\cdot 10^{2}$} &  \multicolumn{2}{c}{$K_2\cdot 10^{4}$}  &  \multicolumn{2}{c}{$K_2\cdot 10^{6}$} \\ \hline \hline
\multirow{6}{*}{$\displaystyle \frac{1}{16}$} & 
\multirow{3}{*}{5E--10}  & $K_1\cdot 10^{-2}$ & 24 & \color{orange}{11} & 38 &\color{orange}{11} & 71 & \color{orange}{11} & 42 & \color{orange}{11}  \\
                               & & $K_1\cdot 10^{-1}$ & 24 & \color{orange}{11} & 38 &\color{orange}{11} & 71 & \color{orange}{11} & 42 & \color{orange}{11}  \\
                               & & $K_1$ & 24 & \color{orange}{11} & 38 &\color{orange}{11} & 71 & \color{orange}{11} & 42 & \color{orange}{11}  \\ \cline{2-11}
& \multirow{3}{*}{1E-8}  & $K_1\cdot 10^{-2}$ & 24 & \color{orange}{11} & 38 &\color{orange}{11} & 71 & \color{orange}{11} & 42 & \color{orange}{11}  \\   
                                     & & $K_1\cdot 10^{-1}$ & 24 & \color{orange}{11} & 38 &\color{orange}{11} & 71 & \color{orange}{11} & 42 & \color{orange}{11}  \\  
                                    & & $K_1$ & 24 & \color{orange}{11} & 38 &\color{orange}{11} & 71 & \color{orange}{11} & 42 & \color{orange}{11}  \\ \hline    
                                    \multirow{6}{*}{$\displaystyle \frac{1}{32}$} & 
\multirow{3}{*}{5E--10}  & $K_1\cdot 10^{-2}$ & 25 & \color{orange}{10} & 45 &\color{orange}{10} & 66 & \color{orange}{10} & 38 & \color{orange}{10}  \\
                               & & $K_1\cdot 10^{-1}$ & 25 & \color{orange}{10} & 45 &\color{orange}{10} & 66 & \color{orange}{10} & 38 & \color{orange}{10}  \\
                               & & $K_1$ & 25 & \color{orange}{10} & 45 &\color{orange}{10} & 66 & \color{orange}{10} & 38 & \color{orange}{10}  \\ \cline{2-11}
& \multirow{3}{*}{1E-8}  & $K_1\cdot 10^{-2}$ & 25 & \color{orange}{10} & 45 &\color{orange}{10} & 66 & \color{orange}{10} & 38 & \color{orange}{10}  \\   
                                     & & $K_1\cdot 10^{-1}$ & 25 & \color{orange}{10} & 45 &\color{orange}{10} & 66 & \color{orange}{10} & 38 & \color{orange}{10}  \\  
                                    & & $K_1$ & 25 & \color{orange}{10} & 45 &\color{orange}{10} & 66 & \color{orange}{10} & 38 & \color{orange}{10}  \\ \hline 
                                     \multirow{6}{*}{$\displaystyle \frac{1}{64}$} & 
\multirow{3}{*}{5E--10}  & $K_1\cdot 10^{-2}$ & 25 & \color{orange}{10} & 57 &\color{orange}{10} & 66 & \color{orange}{10} & 38 & \color{orange}{10}  \\
                               & & $K_1\cdot 10^{-1}$ & 25 & \color{orange}{10} & 57 &\color{orange}{10} & 66 & \color{orange}{10} & 38 & \color{orange}{10}  \\
                               & & $K_1$ & 25 & \color{orange}{10} & 57 &\color{orange}{10} & 66 & \color{orange}{10} & 38 & \color{orange}{10}  \\ \cline{2-11}
& \multirow{3}{*}{1E-8}  & $K_1\cdot 10^{-2}$ & 25 & \color{orange}{10} & 57 &\color{orange}{10} & 66 & \color{orange}{10} & 38 & \color{orange}{10}  \\   
                                     & & $K_1\cdot 10^{-1}$ & 25 & \color{orange}{10} & 57 &\color{orange}{10} & 66 & \color{orange}{10} & 38 & \color{orange}{10}  \\  
                                    & & $K_1$ & 25 & \color{orange}{10} & 57 &\color{orange}{16} & 66 & \color{orange}{10} & 38 & \color{orange}{10}  \\ \hline                             
\end{tabular}
%\end{center}
\end{table}

\begin{table}[hbtp!]
\caption{Number of preconditioned MinRes and fixed-stress splitting iterations for residual reduction by a factor $10^8$ in
the norm induced by the preconditioner when solving the Barenblatt problem where we have redefined $\lambda:=100\cdot\lambda$.}
\label{tab:4}
%\begin{center}
\centering
\begin{tabular}{c|c||c|rr|rr|rr|rr}
$h$ & $\beta$ &   & \multicolumn{2}{c}{$K_2$} &  \multicolumn{2}{c}{$K_2\cdot 10^{2}$} &  \multicolumn{2}{c}{$K_2\cdot 10^{4}$}  &  \multicolumn{2}{c}{$K_2\cdot 10^{6}$} \\ \hline \hline
\multirow{6}{*}{$\displaystyle \frac{1}{16}$} & 
\multirow{3}{*}{5E--10}  & $K_1\cdot 10^{-2}$ & 4 & \color{orange}{2} & 8 &\ \color{orange}{2}  & 16 &  \color{orange}{2}  & 14 &  \color{orange}{2}   \\
                               & & $K_1\cdot 10^{-1}$ & 4 &  \color{orange}{2}  & 8 & \color{orange}{2}  & 16 &  \color{orange}{2}  & 14 &  \color{orange}{2}   \\
                               & & $K_1$ & 4 &  \color{orange}{2}  & 8 & \color{orange}{2}  & 16 & \color{orange}{2}  & 14 &  \color{orange}{2}   \\ \cline{2-11}
& \multirow{3}{*}{1E-8}  & $K_1\cdot 10^{-2}$ & 4 & \color{orange}{2}  & 8 & \color{orange}{2}  & 16 &  \color{orange}{2} & 14 &  \color{orange}{2}   \\   
                                     & & $K_1\cdot 10^{-1}$ & 4 &  \color{orange}{2} & 8 & \color{orange}{2} & 16 &  \color{orange}{2}  & 14 &  \color{orange}{2}   \\  
                                    & & $K_1$ & 4 &  \color{orange}{2}  & 8 & \color{orange}{2}  & 16 &  \color{orange}{2}  & 14 &  \color{orange}{2}  \\ \hline 
                                    \multirow{6}{*}{$\displaystyle \frac{1}{32}$} & 
\multirow{3}{*}{5E--10}  & $K_1\cdot 10^{-2}$ & 6 & \color{orange}{2} & 12 &\ \color{orange}{2}  & 20 &  \color{orange}{2}  & 14 &  \color{orange}{2}   \\
                               & & $K_1\cdot 10^{-1}$ & 6 &  \color{orange}{2}  & 12 & \color{orange}{2}  & 20 &  \color{orange}{2}  & 14 &  \color{orange}{2}   \\
                               & & $K_1$ & 6 &  \color{orange}{2}  & 12 & \color{orange}{2}  & 20 & \color{orange}{3}  & 14 &  \color{orange}{2}   \\ \cline{2-11}
& \multirow{3}{*}{1E-8}  & $K_1\cdot 10^{-2}$ & 6 & \color{orange}{2}  & 12 & \color{orange}{2}  & 20 &  \color{orange}{2} & 14 &  \color{orange}{2}   \\   
                                     & & $K_1\cdot 10^{-1}$ & 6 &  \color{orange}{2} & 12 & \color{orange}{2} & 20 &  \color{orange}{2}  & 14 &  \color{orange}{2}   \\  
                                    & & $K_1$ & 6 &  \color{orange}{2}  & 12 & \color{orange}{2}  & 20 &  \color{orange}{2}  & 14 &  \color{orange}{2}  \\ \hline 
                                      \multirow{6}{*}{$\displaystyle \frac{1}{64}$} & 
\multirow{3}{*}{5E--10}  & $K_1\cdot 10^{-2}$ & 7 & \color{orange}{2} & 16 &\ \color{orange}{2}  & 21 &  \color{orange}{2}  & 14 &  \color{orange}{2}   \\
                               & & $K_1\cdot 10^{-1}$ & 7 &  \color{orange}{2}  & 16 & \color{orange}{2}  & 21 &  \color{orange}{2}  & 14 &  \color{orange}{2}   \\
                               & & $K_1$ & 7 &  \color{orange}{2}  & 16 & \color{orange}{2}  & 21 & \color{orange}{2}  & 14 &  \color{orange}{2}   \\ \cline{2-11}
& \multirow{3}{*}{1E-8}  & $K_1\cdot 10^{-2}$ & 7 & \color{orange}{2}  & 16 & \color{orange}{2}  & 21 &  \color{orange}{2} & 14 &  \color{orange}{2}   \\   
                                     & & $K_1\cdot 10^{-1}$ & 7 &  \color{orange}{2} & 16 & \color{orange}{2} & 21 &  \color{orange}{2}  & 14 &  \color{orange}{2}   \\  
                                    & & $K_1$ & 7 &  \color{orange}{2}  & 16 & \color{orange}{2}  & 21 &  \color{orange}{2}  & 14 &  \color{orange}{2}  \\ \hline                            
\end{tabular}
%\end{center}
\end{table}

\subsection{The four-network model}

This subsection is devoted to the four-network MPET model. As with the previous example, 
the boundary $\Gamma$ of $\Omega$ is split into bottom ($\Gamma_1$), right ($\Gamma_2$), top~($\Gamma_3$), and left ($\Gamma_4$) boundaries. 
The considered boundary conditions are chosen as: 
$$
\begin{array}{rcll}
(\boldsymbol {\sigma} -p_1\boldsymbol I-p_2\boldsymbol I-p_3\boldsymbol I-p_4\boldsymbol I)\boldsymbol n & = &(0,0)^T & \text{ on }\Gamma_1\cup\Gamma_2, \\
(\boldsymbol {\sigma} -p_1\boldsymbol I-p_2\boldsymbol I-p_3\boldsymbol I-p_4\boldsymbol I)\boldsymbol n& = &(0,-1)^T &\text{ on }\Gamma_3, \\
\boldsymbol u & = & \boldsymbol  0 & \text{ on } \Gamma_4, \\
p_1 & = & 2 & \text{ on } \Gamma, \\
p_2 & = & 20 & \text{ on }\Gamma, \\
p_3 & = & 30 & \text{ on } \Gamma, \\
p_4 & = & 40 & \text{ on }\Gamma,
\end{array}
$$
whereas the right hand sides are $\boldsymbol f=\boldsymbol 0$, $g_1=g_2=g_3=g_4=0$. 

Table~\ref{parameters_MPET4} shows the base values of the parameters which have been taken from~\cite{Vardakis2016investigating}. 
%where the four-network MPET model has been used to simulate fluid flow in the human brain. 
The presented numerical results in Table~\ref{tab7} demonstrate again the superiority of the fixed-stress splitting method over the preconditioned MinRes algorithm 
and its robustness with respect to large variations of the coefficients $\lambda$, $K_3$ and $K=K_1=K_2=K_4$.  

%Note that in all three examples, i.e., for the one-network problem, the two-network-problem and the four-network
%problem the observed average residual reduction factors were always below $0.7$ and did not increase as the number of networks
%was increased, which is in accordance with the theoretical findings.
%%
%Moreover, the authors have tried to perform the numerical tests for the parameter ranges leading to the worst results.

\begin{table}[h!]
\caption{Base values of model parameters for a four-network MPET model.}
\label{parameters_MPET4}
%\begin{center}
\centering
\begin{tabular}{ccc}
\hline 
\cellcolor{pinkred!35} 
parameter & 
\cellcolor{pinkred!20} 
value & 
\cellcolor{pinkred!5} 
unit \\[0.0ex] \hline
$\lambda$ & $505$ & Nm$^{-2}$ \\ [0.0ex]
$\mu$ & $216$ & Nm$^{-2}$ \\ [0.0ex]
$c_{p_1}=c_{p_2}=c_{p_3}=c_{p_4}$ & $4.5\cdot10^{-10}$ & m$^2$N$^{-1}$ \\ [0.0ex]
$\alpha_{1}=\alpha_{2}=\alpha_{3}=\alpha_{4}$  & $0.99$ &\\ [0.0ex]
$\beta_{12}=\beta_{24}$ & $1.5\cdot 10^{-19}$ & m$^2$N$^{-1}$s$^{-1}$ \\ [0.0ex]
$\beta_{23}$ & $2.0\cdot 10^{-19}$ & m$^2$N$^{-1}$s$^{-1}$ \\[0.0ex]
$\beta_{34}$ & $1.0\cdot 10^{-13}$ & m$^2$N$^{-1}$s$^{-1}$ \\ [0.0ex]
$K_1=K_2=K_4=K$ & $(1.0 \cdot 10^{-10})/(2.67\cdot10^{-3})$ & m$^2/$Nsm$^{-2}$ \\ [0.0ex]
$K_3$ & $(1.4 \cdot 10^{-14})/(8.9\cdot10^{-4})$ & m$^2/$Nsm$^{-2}$
\end{tabular}
%\end{center}
\end{table}

\begin{table}[htb!]
\caption{Number of preconditioned MinRes and fixed-stress splitting iterations for residual reduction by a factor $10^8$ in
the norm induced by the preconditioner when solving the four-network MPET problem.}
\label{tab7}
\begin{center}
%\resizebox{\columnwidth}{!}{
\begin{tabular}{c|c|c|rr|rr|rr|rr|rr|rr}
$h$ &  &  & \multicolumn{2}{c}{$K_3\cdot 10^{-2}$} &  \multicolumn{2}{c}{$K_3$} &  \multicolumn{2}{c}{$K_3\cdot 10^{2}$} &  \multicolumn{2}{c}{$K_3\cdot 10^{4}$} & 
 \multicolumn{2}{c}{$K_3\cdot 10^{6}$} &  \multicolumn{2}{c}{$K_3\cdot 10^{10}$} \\ [0.1ex] \hline 
 \multirow{9}{*}{$\displaystyle \frac{1}{16}$} & 
\multirow{3}{*}{$\lambda$}  & $K\cdot 10^{-2}$ & 34 &  \color{orange}{10}  & 34 & \color{orange}{10} & 26 & \color{orange}{10} & 23 & \color{orange}{10} & 21 & \color{orange}{10} &  21 & \color{orange}{10}\\ [0.1ex]
                                & & $K$  & 24 & \color{orange}{10} & 24 & \color{orange}{10} & 24 & \color{orange}{10} & 22 & \color{orange}{10}& 21 &\color{orange}{10} & 19 & \color{orange}{10} \\[0.1ex]
                               & & $K\cdot 10^{2}$ & 21 &  \color{orange}{10} & 21 &\color{orange}{10} & 23 & \color{orange}{10} & 23 & \color{orange}{10}& 31 & \color{orange}{10}  & 30 & \color{orange}{10} \\ [0.1ex] \cline{2-15}
& \multirow{3}{*}{$\lambda\cdot 10^4$}  & $K\cdot 10^{-2}$ & 18 &  \color{orange}{2} & 23 &  \color{orange}{2}  & 24 &  \color{orange}{2}  & 34 &  \color{orange}{2}  & 34 &  \color{orange}{2} &  34 &  \color{orange}{2} \\ [0.1ex]
                                    & & $K$  & 11 &  \color{orange}{2}  & 17 &  \color{orange}{2} & 34 &  \color{orange}{2}  & 31 &  \color{orange}{2}  & 31 &  \color{orange}{2}  & 31 &  \color{orange}{2}  \\ [0.1ex]
                                    & & $K\cdot 10^{2}$ & 9 &  \color{orange}{2}  & 14 &  \color{orange}{2}  & 32 &  \color{orange}{2} & 21 &  \color{orange}{2}  & 14 &  \color{orange}{2}   & 14 &  \color{orange}{2}  \\ [0.1ex] \cline{2-15}
& \multirow{3}{*}{$\lambda\cdot 10^8$}  & $K\cdot 10^{-2}$ & 14 &  \color{orange}{2} & 14 &  \color{orange}{2}  & 12 & \color{orange}{2} 
                                  & 12 &  \color{orange}{2}  & 12 & \color{orange}{2}  & 12 &  \color{orange}{2}  \\ [0.1ex]
                                   & & $K$  & 11 &  \color{orange}{2}  & 14 & \color{orange}{2} & 9 &  \color{orange}{2} 
                                   & 7 &  \color{orange}{2} & 7 &  \color{orange}{2}   & 7 &  \color{orange}{2} \\ [0.1ex]
                                  & & $K\cdot 10^{2}$ & 9  &  \color{orange}{2}  & 14 &  \color{orange}{2}  & 9 &  \color{orange}{2}  
                                  & 5 &  \color{orange}{2}  & 5 &  \color{orange}{2}   & 5 &  \color{orange}{2}   \\ [0.1ex] \hline            
\multirow{9}{*}{$\displaystyle \frac{1}{32}$} & 
\multirow{3}{*}{$\lambda$}  & $K\cdot 10^{-2}$ & 34 &  \color{orange}{10}  & 32 & \color{orange}{10} & 26 & \color{orange}{10} & 23 & \color{orange}{10} & 19 & \color{orange}{10} &  19 & \color{orange}{10}\\ [0.1ex]
                                & & $K$  & 24 & \color{orange}{10} & 24 & \color{orange}{10} & 24 & \color{orange}{10} & 22 & \color{orange}{10}& 21 &\color{orange}{10} & 20 & \color{orange}{10} \\[0.1ex]
                               & & $K\cdot 10^{2}$ & 21 &  \color{orange}{10} & 21 &\color{orange}{10} & 21 & \color{orange}{10} & 26 & \color{orange}{10}& 41 & \color{orange}{10}  & 39 & \color{orange}{10} \\ [0.1ex] \cline{2-15}
& \multirow{3}{*}{$\lambda\cdot 10^4$}  & $K\cdot 10^{-2}$ & 18 &  \color{orange}{2} & 25 &  \color{orange}{2}  & 30 &  \color{orange}{2}  & 34 &  \color{orange}{2}  & 34 &  \color{orange}{2} &  34 &  \color{orange}{2} \\ [0.1ex]
                                    & & $K$  & 12 &  \color{orange}{2}  & 20 &  \color{orange}{2} & 35 &  \color{orange}{2}  & 31 &  \color{orange}{2}  & 31 &  \color{orange}{2}  & 31 &  \color{orange}{2}  \\ [0.1ex]
                                    & & $K\cdot 10^{2}$ & 9 &  \color{orange}{2}  & 18 &  \color{orange}{2}  & 34 &  \color{orange}{2} & 21 &  \color{orange}{2}  & 14 &  \color{orange}{2}   & 14 &  \color{orange}{2}  \\ [0.1ex] \cline{2-15}
& \multirow{3}{*}{$\lambda\cdot 10^8$}  & $K\cdot 10^{-2}$ & 14 &  \color{orange}{2} & 14 &  \color{orange}{2}  & 12 & \color{orange}{2} 
                                  & 12 &  \color{orange}{2}  & 12 & \color{orange}{2}  & 12 &  \color{orange}{2}  \\ [0.1ex]
                                   & & $K$  & 12 &  \color{orange}{2}  & 14 & \color{orange}{2} & 9 &  \color{orange}{2} 
                                   & 7 &  \color{orange}{2} & 7 &  \color{orange}{2}   & 7 &  \color{orange}{2} \\ [0.1ex]
                                  & & $K\cdot 10^{2}$ & 11  &  \color{orange}{2}  & 14 &  \color{orange}{2}  & 9 &  \color{orange}{2}  
                                  & 6 &  \color{orange}{2}  & 5 &  \color{orange}{2}   & 5 &  \color{orange}{2}   \\ [0.1ex] \hline        
\multirow{9}{*}{$\displaystyle \frac{1}{64}$} & 
\multirow{3}{*}{$\lambda$}  & $K\cdot 10^{-2}$ & 34 &  \color{orange}{10}  & 32 & \color{orange}{10} & 26 & \color{orange}{10} & 21 & \color{orange}{10} & 19 & \color{orange}{10} &  19 & \color{orange}{10}\\ [0.1ex]
                                & & $K$  & 24 & \color{orange}{10} & 24 & \color{orange}{10} & 24 & \color{orange}{10} & 23 & \color{orange}{10}& 22 &\color{orange}{10} & 21 & \color{orange}{10} \\[0.1ex]
                               & & $K\cdot 10^{2}$ & 21 &  \color{orange}{10} & 21 &\color{orange}{10} & 21 & \color{orange}{10} & 36 & \color{orange}{10}& 45 & \color{orange}{10}  & 45 & \color{orange}{10} \\ [0.1ex] \cline{2-15}
& \multirow{3}{*}{$\lambda\cdot 10^4$}  & $K\cdot 10^{-2}$ & 20 &  \color{orange}{2} & 28 &  \color{orange}{2}  & 34 &  \color{orange}{2}  & 34 &  \color{orange}{2}  & 34 &  \color{orange}{2} &  34 &  \color{orange}{2} \\ [0.1ex]
                                    & & $K$  & 13 &  \color{orange}{2}  & 25 &  \color{orange}{2} & 36 &  \color{orange}{2}  & 31 &  \color{orange}{2}  & 31 &  \color{orange}{2}  & 31 &  \color{orange}{2}  \\ [0.1ex]
                                    & & $K\cdot 10^{2}$ & 6 &  \color{orange}{2}  & 25 &  \color{orange}{2}  & 36 &  \color{orange}{2} & 21 &  \color{orange}{2}  & 14 &  \color{orange}{2}   & 14 &  \color{orange}{2}  \\ [0.1ex] \cline{2-15}
& \multirow{3}{*}{$\lambda\cdot 10^8$}  & $K\cdot 10^{-2}$ & 14 &  \color{orange}{2} & 14 &  \color{orange}{2}  & 12 & \color{orange}{2} 
                                  & 12 &  \color{orange}{2}  & 12 & \color{orange}{2}  & 12 &  \color{orange}{2}  \\ [0.1ex]
                                   & & $K$  & 12 &  \color{orange}{2}  & 14 & \color{orange}{2} & 9 &  \color{orange}{2} 
                                   & 7 &  \color{orange}{2} & 7 &  \color{orange}{2}   & 7 &  \color{orange}{2} \\ [0.1ex]
                                  & & $K\cdot 10^{2}$ & 12  &  \color{orange}{2}  & 14 &  \color{orange}{2}  & 9 &  \color{orange}{2}  
                                  & 6 &  \color{orange}{2}  & 5 &  \color{orange}{2}   & 5 &  \color{orange}{2}   \\ [0.1ex] \hline                                                                 
\end{tabular}
%}
\end{center}
\end{table}

\section{Concluding remarks}\label{conclusions}

To the best of our knowledge, this paper is the first example of a proposed and analyzed fixed-stress splitting scheme for a three-field formulation of the MPET model. 
Fundamental to the linear 
convergence of the evolved algorithm is the incorporation of stabilization that employs the sum of all pressures. 
By applying the stability results proven in~\cite{Hong2018conservativeMPET}, 
we have demonstrated that the contraction rate of this fixed-point iteration is independent of any model physical parameters. 
Furthermore, the performed numerical experiments have clearly demonstrated the efficiency of the presented fixed-stress scheme along with its superiority over a fully implicit scheme 
which utilizes a norm-equivalent preconditioner.

\bibliographystyle{plain}
\bibliography{reference_mpet}

\end{document}